\documentclass[12pt]{article}
\usepackage[english]{babel}
\usepackage{amsmath}
\usepackage{amssymb}
\usepackage{dsfont}
\usepackage{color}
\usepackage{graphicx}
\usepackage{amsthm}
\usepackage{algorithm,algorithmic}
\usepackage{subcaption}
\usepackage{natbib}
\usepackage{amsthm}
\usepackage{booktabs}
\usepackage{soul}
\usepackage{thmtools}
\usepackage{threeparttable}
\usepackage{url}
\usepackage{float}
\usepackage{enumerate}
\usepackage{etex}
\usepackage{amsmath,latexsym}
\usepackage{amsthm}
\usepackage{graphicx}
\usepackage[linktoc=all]{hyperref}
\usepackage{titlesec}
\usepackage{makeidx}
\usepackage{tikz}
\usepackage{pgfplots}
\usepackage{xspace}

\allowdisplaybreaks

\hypersetup{
    colorlinks,
    citecolor=black,
    filecolor=black,
    linkcolor=black,
    urlcolor=black
}

\definecolor{Darkgreen}{rgb}{0,0.6,0}

\renewcommand{\bibname}{References}
\newtheorem{ex}{Example}

\newtheorem{pro}{Proposition}

\newtheorem{coro}{Corollary}

\def\r{\hat{\rho}}

\newcommand{\BPC}{\textbf{B\&P\&C\xspace}}

\title{
An extended version of a Branch-Price-and-Cut Procedure for the Discrete Ordered Median Problem}
\begin{document}

\author{\textbf{Samuel Deleplanque}\\
Ifsttar, COSYS, ESTAS, Universit\'e Lille Nord de France,\\
\textbf{Martine Labb\'e}\\
D\'epartament d'Informatique, Facult\'e des Sciences,\\ Universit\'e Libre de Bruxelles,\\
\textbf{Diego Ponce}\\
Instituto de Matem\'aticas de la Universidad de Sevilla (IMUS),
\\
\textbf{ Justo Puerto}\\
Instituto de Matem\'aticas de la Universidad de Sevilla (IMUS).
 \\}
\maketitle

\begin{abstract}
The Discrete Ordered Median Problem (DOMP) is formulated as a set partitioning problem using an exponential number of variables. Each variable corresponds to a set of demand points allocated to the same facility with the information of the sorting position of their corresponding costs. We develop a column generation approach to solve the continuous relaxation of this model. Then, we apply a branch-price-and-cut algorithm to solve to optimality small to moderate size of DOMP in competitive computational time.
\end{abstract}

\section{Introduction \label{section:1}}

Logistics is a new most active field in nowadays Operations Research and Location Analysis is among its most important building blocks. Motivated by the need of applying more flexible models in Logistics, in the last years, a new family of location models, namely  the Ordered Median location Problem has been proposed. An ordered median objective function computes ordered weighted averages of vectors (\cite{Nickel2005}) and when it is applied to location problems those vectors are distances or allocation costs from clients to service facilities. Ordered median location problems were first introduced in networks and continuous spaces by \cite{NP1999} and \cite{PF2000}, respectively. Later, they were extended to the discrete setting by  \cite{Nickel2001,Boland2006}. The Discrete Ordered Median Problem (DOMP) has been widely studied since the 90's and there is a number of different formulations,  solution approaches and applications available in the literature (\citet{Boland2006,Dominguezmarin2003, Marin2009, Marin2010, Nickel2001, NP1999, Nickel2005, PP2013, Puerto2008, Puerto2009, Puerto2014}).

Given a set of clients and a set of candidate locations and assuming that the allocation costs of clients to facilities are known, DOMP consists in choosing $p$ facility locations and assigning each client to a facility with smallest allocation cost in order to minimize the ordered weighted average of these costs. The ordered weighted average sorts the allocation costs in a non-decreasing sequence and then it performs the scalar product of this so-obtained sorted cost vector with a given vector of weights.

There are several valid formulations for DOMP that exploit specific features of the problem (see e.g. \cite{Boland2006,Marin2009,Labbe2017} and the references therein). In \cite{Labbe2017} a new formulation for DOMP has been proposed, based on a set packing approach, that is valid for general cost coefficient. This formulation  gives rise to rather tight integrality gaps and was shown to be reasonably efficient to solve medium size instances when embedded in a branch-and-cut (B\&C) scheme. In this paper we explore a different paradigm for solving DOMP based on an extended formulation using an exponential number of variables corresponding to a set partitioning model. Each variable represents a set of couples (client, position). These clients are served by the same facility and their position indicates the situation of this allocation cost in the sorted list of allocation costs in any feasible solution.  To handle the exponential number of variables we use a column generation approach that is embedded in a branch-price-and-cut (\BPC) algorithm. A recent similar approach can be seen in \cite{Hossein2016}. This scheme has never been applied to DOMP and it opens new avenues of research. Therefore, the contribution of this paper is to propose a new perspective in the resolution of DOMP based on formulations with an exponential number of variables and to develop an efficient \BPC \; algorithm to handle them.

This paper is organized as follows. After the introduction, Section \ref{section:2.2} introduces a new set partitioning formulation for DOMP. This formulation uses an exponential number of variables where each element of the partition is a set of clients together with their sorted positions that are assigned to the same server. This formulation is theoretically compared in Section \ref{section:2.3} with another valid formulation described in Section \ref{section:2.1} borrowed from \cite{Labbe2017}. Section \ref{section:2.4} describes the column generation algorithm that we have designed to overcome the large number of variables in the model. We prove that the pricing subproblem is solvable efficiently in polynomial time by using an \textit{ad hoc} dynamic programming algorithm. We devote our Section \ref{section:3} to determine the implementations details of our  \BPC \; algorithm.   We develop a GRASP heuristic, in Section \ref{section:3.1}, that is used both to generate a promising initial solution and a pool of variables to initialize the column generation routine.
We also develop a stabilization routine, based in \cite{Pessoa2010}, that reduces considerably the number of iterations of the column generation approach in Section \ref{section:3.2}. In addition, sections \ref{section:3.3} and \ref{section:3.4} are devoted to present two additional improvements, namely a pricer heuristic and a preprocessing. The next two subsections, \ref{section:3.5} and \ref{section:3.6}, present our branching strategies and some families of valid inequalities that will be added to the branch-and-price algorithm. The next section, namely Section \ref{section:4} is devoted to report on the final computational experiments of this paper. Here, we report on  the performance of the solution approach. Besides, we also compare the  performance of the \BPC\;  algorithm presented in this paper against the  compact formulation in Section \ref{section:2.1}. The paper ends with a section devoted to concluding remarks.

\section{Problem definition and formulations\label{section:2}}

Let $I$ be a set of $n$ points which at the same time represent clients and potential facility locations which are assumed to be uncapacitated; and let $c_{ij}$ denote the cost for serving client $i$'s demand from facility $j$.

\renewcommand{\labelenumi}{\arabic{enumi}.}

Given a set $J$ of $p$ open facilities, let $c_i(J)$ represents the cost for allocating client $i$ to the cheapest  facility in $J$ so that $c_i(J):=\displaystyle\min_{j \in J} c_{ij}$.

Now let us sort the costs $c_i(J)$, $i\in I$ by non-decreasing order of their values. The elements of the resulting vector of ordered costs are denoted by $c^{(k)}(J)$ and satisfy
$c^{(1)}(J)\leq \cdots \leq c^{(n)}(J).$

Given vector $\lambda=(\lambda^k)^n_{k=1}$ satisfying $\lambda^k\geq 0, k=1,\dots,n$, the objective function of DOMP, is defined as
\begin{equation}\label{fo:def}z(J):=\sum_{k=1}^n\lambda^kc^{(k)}(J).\end{equation}

Recall that this objective function provides a very general paradigm to encompass standard and new location models. For instance, if $\lambda^1=\dots=\lambda^n=1$ we obtain the median objective, if $\lambda^1=\lambda^2=\dots=\lambda^{n-1}=0,\lambda^n=1$ we obtain the center objective, if $\lambda^1=\lambda^2=\dots=\lambda^{n-1}=\alpha,\lambda^n=1$, where $\alpha= [0,1]$, we obtain a convex combination of median and center objectives (centdian), etc.

The $p$-facility Discrete Ordered Median Problem looks for the subset $J$ of $p$ facilities to open in order to minimize the ordered median function:
\begin{equation}\label{of:DOMPdefiniton} \tag{DOMP}\min_{J\subseteq I:|J|=p}z(J).\end{equation}

There are several available formulations of DOMP in the literature using different spaces of variables. Among them we mention those based on some combinations of the $p$-median and permutation polytopes (\citep{Boland2006}) or on   coverage approaches based on radius variables (\citep{Puerto2008}, \citep{Marin2009,Marin2010}).

\subsection{An explicit formulation for DOMP: The Weak Order Constraints\label{section:2.1}}

In  the following, we recall the Weak Order Constraints formulation, that we will refer to as $WOC$, introduced in \citet{Labbe2017}, that will be the starting point for the developments presented in this paper. This formulation uses two types of binary variables. Variables $y_j$ assume value 1 if facility $j$ is open (i.e. $j \in J$) and 0 otherwise. Variables  $x_{ij}^k$ are equal to 1 if client $i$ is allocated to facility $j$ and the corresponding cost occupies position $k$ in the allocation cost ranking (i.e. $c^{(k)}(J)=c_{ij}$). The choice of this formulation is motivated by its good performance in terms of integrality gap (see \citep{Labbe2017}). However, it requests important memory space since it needs $O(n^3)$ binary variables which may become prohibitive for moderate $n$.

Let  $R=(r_{ij}) \in \mathbb{N}^{n\times n}$ be a matrix such that $r_{ij}=\ell$ if $c_{ij}$ is the $\ell$-th element in the sorted list of the costs in $C=(c_{ij})$, where ties are broken arbitrarily. In other words, $r_{ij}$ is the position in the above list of the allocation cost $c_{ij}$ of the problem. For the sake of readability the reader is referred to Example \ref{ex:firstsolution} in Section \ref{section:2.4}. Thus, the formulation is

\begin{eqnarray}
(WOC):\;  \min&\displaystyle\sum_{i=1}^n\sum_{j=1}^n\sum_{k=1}^n\lambda^kc_{ij}x_{ij}^k\label{c3:ofdomp4}\\
\mbox{s.t.}&\displaystyle\sum_{j=1}^n\sum_{k=1}^nx_{ij}^k=1&i=1,\dots,n\label{eq:1fb}\\
&\displaystyle\sum_{i=1}^n\sum_{j=1}^nx_{ij}^k = 1&k=1,\dots,n\label{eq:2fb}\\
&\displaystyle\sum_{k=1}^nx_{ij}^k \leq y_j & {\small i,j=1,\dots,n} \quad \label{in:1fb}\\
&\displaystyle\sum_{j=1}^ny_{j} = p\label{eq:3fb}\\
&\hspace*{-4cm}\displaystyle \sum_i^n\sum_j^n\left(\sum_{i'=1}^n\sum_{\substack{j'=1:\\ r_{i'j'}\leq r_{ij}}}^nx_{i'j'}^k + \sum_{i'=1}^n\sum_{\substack{j'=1:\\r_{i'j'}\geq r_{ij}}}^nx_{i'j'}^{k-1}\right) \le n^2, & k=2,\cdots,n\label{in:3f1}\\
&x_{ij}^k \in \{0,1\}& i,j,k=1,\dots,n \quad \label{binary:xf1}\\
&y_j \in \{0,1\}&  j=1,\dots,n,\label{binary:yf1}
\end{eqnarray}

By means of (\ref{eq:1fb}) we ensure that each location is served by exactly one facility. In the same way, in each position there must be exactly one allocation cost (\ref{eq:2fb}). We know that a client can be allocated to a facility only if this facility is open, i.e. $x_{ij}^k\leq y_j$ for all $i,j,k$. Furthermore, each allocation cost of a client to a facility can be placed in at most one position. Hence, $x_{ij}^k\leq y_j$ can be strengthened yielding constraints (\ref{in:1fb}). The equality constraint (\ref{eq:3fb}) implies that there are exactly $p$ open facilities.

The constraints (\ref{in:3f1}), called \textit{weak order constraints}, ensure that if client $i$ allocated to facility $j$, occupies the $k$-th position in the client ranking then in $(k-1)$-th position there must be a more preferred allocation cost. This property is enforced by the coefficients of each variable in the inequality. In each constraint there are two different positions, $k$ and $k-1$, so that, by (\ref{eq:2fb}), only two variables must take value one and all the others will be equal to zero. If we do not take into account the variables assuming the value zero and we assume that the variables with value one for positions $k$ and $k-1$ correspond to allocation pairs in sorted  position $s$ and $t$, respectively, the inequality reduces to the following expression:
$$(n^2-(s-1))x_{i_sj_s}^k+tx_{i_tj_t}^{k-1}\le n^2,$$
which is valid if and only if $t<s$.

Finally, the variables are binary, see (\ref{binary:xf1}) and (\ref{binary:yf1}).

$WOC$ can be reinforced by adding some valid inequalities
\begin{equation}
\sum_{i'=1}^n\sum_{\substack{j'=1:\\ r_{i'j'}\leq r_{ij}}}^nx_{i'j'}^k + \sum_{i'=1}^n\sum_{\substack{j'=1:\\ r_{i'j'}\geq r_{ij}}}^nx_{i'j'}^{k-1} \le 1, \;  i,j=1,\cdots,n,\; k=2,\cdots,n. \label{in:3f1-des}
\end{equation}
Observe that constraints (\ref{in:3f1}) are the aggregation over $i,j$ of inequalities (\ref{in:3f1-des}). These inequalities are the so called \textit{strong order constraints}, see \cite{Labbe2017} for a detailed explanation.

\subsection{A set partitioning formulation\label{section:2.2}}

From a linear programming relaxation point of view the above formulation is not the strongest one but it provides a good  compromise between the number of required constraints and the quality of its linear relaxation bound, see  \cite{Labbe2017}. Further, it
allows to solve to optimality problems of moderate size. One of its drawbacks is the use of a cubic number of variables, which can be prohibitive for large $n$.  A second  important problem of most known formulations for DOMP is the high degree of symmetry in case of allocation costs ($C$) or weighted ordered vector ($\lambda$) with many ties.

The reasons above motivate the introduction of a new formulation based on a different rationale. We observe that a solution for DOMP is a partition of the clients together with their positions in the sorted vector of costs so that each  subset of clients in the partition is allocated to the same facility.

Let us consider sets  of couples $(i,k)$ where the first component refers to client $i$ and the second to position $k$, namely $S=\{(i,k): \text{ for some } i,k=1,\dots,n\}$.
Associated with each set $S$ and facility $j$, we define variables

\begin{eqnarray*}y_{S}^j&=&\left\{\begin{array}{cl}1&\text{if set $S$ is part of a feasible solution, i.e. } (i,k)\in S \text{ iff } x_{ij}^k=1\\0&\text{otherwise.}\end{array}\right.\end{eqnarray*}

We observe that in any feasible solution each client $i$ must occupy a unique sorted position $k$ and must be allocated to a unique facility $j$, thus the following relationship holds $x_{ij}^k=\sum_{S\ni(i,k)}y_S^j$, for all $i,j,k$.

Next, assuming that all clients in $S$  are allocated to facility $j$ and that the positions that appear  in the second entry of the couples $(i,k)$ of the set $S$ satisfy the sorting among their allocation costs, i.e. $c_{ij}\le c_{i'j}$ whenever $(i,k)$, $(i',k')\in S$ and $k < k'$, we can evaluate the cost $c_S^j$ induced by the set $S$ provided that its clients are assigned to facility $j$ in a feasible solution:

\begin{equation}c_S^j=\sum_{(i,k)\in S}\lambda^kc_{ij}.\label{eq:relation}\end{equation}
To simplify the presentation in the following we denote by $(i,\cdot)$  the couples whose first entry is $i$ regardless of the value of the second entry. Analogously, $(\cdot,k)$ denotes the couples whose second entry is $k$ regardless of the value of the first entry.

We give next a valid formulation for DOMP using the set of variables $y_S^j$. This will be our Master Problem ($MP$) in Section \ref{section:2.2}.

\begin{eqnarray}
\textbf{(MP) }\min&\displaystyle\sum_{j=1}^n\sum_Sc_S^jy_S^j&\label{of:ma}\\
s.t.&\displaystyle\sum_{j=1}^n\sum_{S\ni(i,\cdot)}y_S^j&=1,\forall\,i\label{eq:1ma}\\
&\displaystyle\sum_{j=1}^n\sum_{S\ni(\cdot,k)}y_S^j&=1,\forall\,k\label{eq:2ma}\\
&\displaystyle\sum_{S}y_S^j&\le1,\forall\,j\label{eq:3ma}\\
&\displaystyle\sum_{j=1}^n\sum_{S}y_S^j&\le p,\label{eq:4ma}\\
&\hspace*{-1cm}\displaystyle\sum_{i=1}^n\sum_{j=1}^n\left(\sum_{\substack{S\ni(i',k)\\:r_{i' j'}\le r_{ij}}}y_S^{ j'}+\sum_{\substack{S\ni(i',k-1)\\:r_{i' j'}\ge r_{ij}}}y_S^{ j'}\right)&\le n^2,k=2,\dots,n\label{eq:5ma}\\
&y_S^j&\in\{0,1\},\forall\,S,j,
\end{eqnarray}

The objective function (\ref{of:ma}) accounts for the sorted weighted cost of any feasible solution. Constraints (\ref{eq:1ma}) ensure that each client appears exactly once in a set $S$. Constraints (\ref{eq:2ma})  ensure that each position is taken exactly once by a client in a set $S$. Constraints (\ref{eq:3ma}) guarantees that each facility $j$ serves at most one set $S$ of clients. Inequality (\ref{eq:4ma})  states that at most $p$ facilities will be opened. By the following family of inequalities (\ref{eq:5ma})  we enforce the correct sorting of the costs in any feasible solution. Finally, the variables are binary.  We note in passing that this formulation is not a Dantzig-Wolfe reformulation of $WOC$ but a new formulation based on the properties of the problem. Indeed, the definition of a column $y_S^j$ includes conditions on the position of the clients in $S$. Hence partial order constraints are transfered to the pricing problem.

The above formulation can be strengthen by adding valid inequalities borrowed from $WOC$. Indeed, one can  translate valid inequalities (\ref{in:3f1-des}) in terms of the $y_S^j$ variables so that they can be used in the set partition formulation of DOMP. The translation of (\ref{in:3f1-des}) results in:
\begin{equation}
\displaystyle\sum_{\substack{S\ni(i',k)\\:r_{i' j'}\le r_{ij}}}y_S^{ j'}+\sum_{\substack{S\ni(i',k-1)\\:r_{i' j'}\ge r_{ij}}}y_S^{ j'}\le 1,\; i,j=1,\dots,n,k=2,\dots,n.\label{cuts}
\end{equation}

\subsection{Theoretical comparison of formulations \label{section:2.3}}
One can prove that the linear relaxation of $MP$, from now on $LRMP$, is tighter than that of $WOC$. Let $P_{MP}$ and $P_{WOC} $,   denote, respectively, the polyhedra defined by the feasible domains of $MP$ and $WOC$ relaxing the integrality constraints. Moreover, let $N$ be the dimension of the space of variables $y_S^j$ defined above and consider the following mapping
$$\begin{array}{rcl}
f:[0,1]^N&\longrightarrow& [0,1]^{n^3}\times[0,1]^n\\
(y_s^j)&\longmapsto&(x_{ij}^k,y_j)
\end{array}$$
defined by the following two equations
\begin{equation}
x_{ij}^k=\sum_{S\ni (i,k)}y_S^j\quad i,j,k=1,\dots,n\label{c3r:1yx}
\end{equation}
and
\begin{equation}
y_j=\sum_{S}y_S^j\quad j=1,\dots,n\label{c3r:2yy}.
\end{equation}
\begin{pro}
\label{c3-pro-inclusionMP-DOMP}
Let $p=(y_S^j)$ if $p\in P_{MP} $ then $f(p)\in P_{WOC}$.
\end{pro}
\begin{proof}
Let us assume that $p\in P_{MP} $. We prove that $f(p)$ satisfies (\ref{eq:1fb})-(\ref{in:3f1}). To prove $(\ref{eq:1fb})$, observe that, according to the definition of $x_{ij}^k$ in (\ref{c3r:1yx}), $\sum_{S\ni (i,.)} y_S^j=\sum_{k=1}^n x_{ij}^k$. Therefore, substituting in (\ref{eq:1ma}) we get the desired result. Checking the validity of (\ref{eq:2fb}) is analogous.

Now, we prove (\ref{in:1fb}). Observe that by (\ref{c3r:1yx}) $\sum_{k=1}^n x_{ij}^k = \sum_{k=1}^n \sum_{S\ni (i,k)} y_S^j=\sum_{S\ni (i,\cdot)} y_S^j$ and then
$$ \sum_{S\ni (i,\cdot)} y_S^j \le  \sum_{k=1}^n \sum_{S} y_S^j\le 1.$$

This last inequality holds by (\ref{eq:3ma}) which proves (\ref{in:1fb}). To check (\ref{eq:3fb}) we replace (\ref{c3r:2yy})  on (\ref{eq:4ma}) to obtain $\sum_{j=1}^n y_j\le 1$. The equality follows because setting extra $y_j$ variables to 1 do not worsen the objective function since all $y_j$ variables have null cost. Finally, (\ref{eq:5ma}) follows analogously substituting (\ref{c3r:1yx}) in (\ref{in:3f1}).
\end{proof}

Hence, it is clear that the bound obtained by LRMP is at least as good as the bound provided by the linear relaxation of $WOC$. There are instances where the inclusion is strict as shown by the integrality gap results reported in Table \ref{ResultsLP20and30}.

Let $P_{SOC}$ be the polyhedron defined by the constraints (\ref{eq:1fb})-(\ref{eq:3fb}) and (\ref{in:3f1-des}) assuming the variables $(x,y)\in [0,1]^{n^3}\times [0,1]^n$. Observe that this is the polyhedron that results  from $P_{WOC}$ by replacing (\ref{in:3f1}) by (\ref{in:3f1-des}). Analogously, let $P_{SMP}$ be the  convex polyhedron defined by the constraints (\ref{eq:1ma})-(\ref{eq:4ma}) and (\ref{cuts}), that results from $P_{MP}$ replacing (\ref{eq:5ma}) by (\ref{cuts}).  We assume variables $y\in [0,1]^{N}$.
The following results relates the feasible solutions of the linear relaxations of $MP$ and $WOC$ whenever all the cuts coming from the \textit{strong order constraints} are added to both formulations.

\begin{coro}
Let $p=(y_S^j)$ if $p\in P_{SMP}$ then $f(p)\in P_{SOC}$.
\end{coro}
The proof is similar to that of Proposition \ref{c3-pro-inclusionMP-DOMP}.

\subsection{Column generation to solve LRMP}\label{section:2.4}

Due to the fact that $MP$ can have a number of variables too large to be handled directly, in this section we describe a column generation approach to solve it.

We begin by obtaining the dual of LRMP. In order to do that let ($\alpha, \beta, \gamma,\delta,\epsilon$) be the dual variables associated, respectively, to constraints (\ref{eq:1ma}), (\ref{eq:2ma}), (\ref{eq:3ma}), (\ref{eq:4ma}) and (\ref{eq:5ma}). Then, DP, the dual problem of LRMP is

\begin{align}
\textbf{(DP)}\max&\displaystyle\sum_{i=1}^n\alpha_i+\sum_{k=1}^n\beta_k-\sum_{j=1}^n\gamma_j-p\delta-\sum_{k=2}^nn^2\epsilon_k&&\\
s.t.&\displaystyle\sum_{\substack{i=1\\:(i,\cdot)\in S}}^n\alpha_i+\sum_{\substack{k=1\\:(\cdot,k)\in S}}^n\beta_k-\gamma_j-\delta\nonumber\\
&\displaystyle-\sum_{k=2}^n\sum_{i'=1}^n\sum_{ j'=1}^n\left(\sum_{\substack{(i,k)\in S\\:r_{i' j'}\ge r_{ij}}}\epsilon_k+\sum_{\substack{(i,k-1)\in S\\:r_{i' j'}\le r_{ij}}}\epsilon_{k}\right)\le c_S^{j},\quad \forall\,j,S\\
&\gamma_j\ge0,\quad  \forall\,j\nonumber\\
&\delta\ge0,\quad \nonumber\\
&\epsilon_k\ge0\quad \forall\,k.\nonumber
\end{align}

In order to apply the column generation procedure let us assume that we are given a set of columns that defines a  restricted  linear relaxation of the Master Problem, from now on $ReLRMP$. This problem is solved to optimality and we get its dual optimal variables ($\alpha^*, \beta^*, \gamma^*,\delta^*,\epsilon^*$). See Example \ref{ex:firstsolution}. The reduced cost, $\overline c_S^j$, of the column $y_S^j$, namely $\overline c_S^j=c_S^j-z_S^j$ is given as:
$$\overline c_S^j=c_S^j+\gamma_j^*+\delta^*+\sum_{k=2}^n\sum_{i'=1}^n\sum_{ j'=1}^n\left(\sum_{\substack{(i,k)\in S\\:r_{i' j'}\ge r_{ij}}}\epsilon_k^*+\sum_{\substack{(i,k-1)\in S\\:r_{i' j'}\le r_{ij}}}\epsilon_{k}^*\right)-\sum_{\substack{i=1\\:(i,\cdot)\in S}}^n\alpha_i^*-\sum_{\substack{k=1\\:(\cdot,k)\in S}}^n\beta_k^*.$$

If $\overline c_S^j\ge0$ for all $S,j$ the current solution of ReLRMP is also optimal for the LRMP and the column generation procedure is finished.

Otherwise, one has identified one (some) new column(s) to be added to the current reduced master problem to proceed further. In each iteration, the ReLRMP and its reduced costs provide lower and upper bounds  for the LRMP. Indeed it holds (\citet{Desrosiers2005})

\begin{eqnarray}
z_{ReLRMP}+ p\cdot\min_{j,S}\overline c_S^j\le z_{LRMP}\le z_{ReLRMP},\label{lb1}\\
z_{ReLRMP}+ \sum_{j=1}^n\min_{S}\overline c_S^j\le z_{LRMP}\le z_{ReLRMP}. \label{lb2}
\end{eqnarray}
where $z_{ReLRMP}$ and $z_{LRMP}$ denote the optimal value of $ReLRMP$ and $LRMP$ respectively.

\begin{ex}\label{ex:firstsolution}

Consider the following cost matrix:
$$C=\left(\begin{array}{ccc}1&3&6\\3&1&8\\6&8&1\end{array}\right)$$
and the vector $\lambda=(4,2,1)$ The precedence matrix is the following
$$R=\left(\begin{array}{ccc}1&4&6\\5&2&8\\7&9&3\end{array}\right).$$
For $n=3$, there are 33 different sets of couples $(i,k)$.
\begin{eqnarray*}
\begin{array}{l}
S_1=\{(1,1)\}\\
S_2=\{(1,2)\}\\
S_3=\{(1,3)\}\\
S_4=\{(2,1)\}\\
S_5=\{(2,2)\}\\
S_6=\{(2,3)\}\\
S_7=\{(3,1)\}\\
S_8=\{(3,2)\}\\
S_9=\{(3,3)\}\\
S_{10}=\{(1,1),(2,2)\}\\
S_{11}=\{(1,1),(2,3)\}
\end{array}
&
\begin{array}{l}
S_{12}=\{(1,1),(3,2)\}\\
S_{13}=\{(1,1),(3,3)\}\\
S_{14}=\{(1,2),(2,1)\}\\
S_{15}=\{(1,2),(2,3)\}\\
S_{16}=\{(1,2),(3,1)\}\\
S_{17}=\{(1,2),(3,3)\}\\
S_{18}=\{(1,3),(2,1)\}\\
S_{19}=\{(1,3),(2,2)\}\\
S_{20}=\{(1,3),(3,1)\}\\
S_{21}=\{(1,3),(3,2)\}\\
S_{22}=\{(2,1),(3,2)\}
\end{array}
&
\begin{array}{l}
S_{23}=\{(2,1),(3,3)\}\\
S_{24}=\{(2,2),(3,1)\}\\
S_{25}=\{(2,2),(3,3)\}\\
S_{26}=\{(2,3),(3,1)\}\\
S_{27}=\{(2,3),(3,2)\}\\
S_{28}=\{(1,1),(2,2),(3,3)\}\\
S_{29}=\{(1,1),(2,3),(3,2)\}\\
S_{30}=\{(1,2),(2,1),(3,3)\}\\
S_{31}=\{(1,2),(2,3),(3,1)\}\\
S_{32}=\{(1,3),(2,1),(3,2)\}\\
S_{33}=\{(1,3),(2,1),(3,2)\}.\\
\end{array}
\end{eqnarray*}
We consider as initial pool of columns the variables $y_{18}^1$ and $y_8^3$. With this set of variables, the ReLRMP is
$$
\begin{array}{rrrll}
\textbf{(ReLRMP)}\min&+2y_{5}^2&+10y_{13}^1\\
s.t.&&+y_{13}^1&\ge1&i=1\\
&+y_{5}^2&&\ge1&i=2\\
&&+y_{13}^1&\ge1&i=3\\
&&+y_{13}^1&\ge1&k=1\\
&+y_{5}^2&&\ge1&k=2\\
&&+y_{13}^1&\ge1&k=3\\
&&-y_{13}^1&\ge-1&j=1\\
&-y_5^2&&\ge-1&j=2\\
&&&\ge-1&j=3\\
&-y_5^2&-y_{13}^1&\ge-2&\\
&-8y_5^2&-y_{13}^1&\ge-9&k=2\\
&-2y_5^2&-3y_{13}^1&\ge-9&k=3\\
&&y&\ge0&
\end{array}
$$
Actually, we are interested in its dual problem:
$${\scriptsize
\begin{array}{rrrrrrrrrrrrrll}
\textbf{(DP)}\max&+\alpha_1&+\alpha_2&+\alpha_3&+\beta_1&+\beta_2&+\beta_3&-\gamma_1&-\gamma_2&-\gamma_3&-2\delta&-9\epsilon_2&-9\epsilon_3\\
s.t.&&+\alpha_2&&&+\beta_2&&&-\gamma_2&&-\delta&-8\epsilon_2&-2\epsilon_3&\le2&(y_5^2)\\
&+\alpha_1&&+\alpha_3&+\beta_1&&+\beta_3&-\gamma_1&&&-\delta&-\epsilon_2&-3\epsilon_3&\le10&(y_{13}^1)\\
\end{array}
}
$$
$$\alpha,\beta,\gamma,\delta,\epsilon\ge0$$
Solving (DP)  the solution is $\alpha_2=2,\beta_3=10$ and the value of the objective function is $f=12$.
\end{ex}

\subsection{Solving the pricing subproblem\label{section:2.5}}

Although any column $y_S^j$ with negative reduced cost may be added to ReLRMP, we will follow a strategy that identifies the most negative reduced cost for each facility $j$. This approach may give rise to several candidate columns (multiple pricing, see \citet{Chvatal1983}), which is advantageous for this procedure.

In order to do that, we solve for  each facility $j$ a subproblem to find the column with minimum reduced cost associated with a feasible set $S$, namely a solution that satisfies that there is at most one pair $(i,\cdot)$ for each client $i$ and one pair $(\cdot,k)$ for each position $k$. Furthermore, the set $S$ must enjoy that the allocation costs of its couples are ranked accordingly. We solve this problem by the following dynamic programming algorithm. The reader may gain some intuition interpreting the algorithm as a shortest path in a graph built upon the matrix $D_j$ defined in (\ref{matrizdj}).

Let $d_{ij}^k$ be the contribution of the pair $(i,k)$ to the reduced cost of any column $y_S^j$ such that $(i,k)\in S$. Depending on the values of $k$,  $d_{ij}^k$ is given by
$$d_{ij}^k=\left\{\begin{array}{ll}\displaystyle\lambda^kc_{ij}+\sum_{i'=1}^n\sum_{\substack{ j'=1\\:r_{i' j'}\le r_{ij}}}\epsilon_{k+1}-\alpha_i-\beta_k&\text{if }k=1,\\
\displaystyle\lambda^kc_{ij}+\sum_{i'=1}^n\sum_{\substack{ j'=1\\:r_{i' j'}\ge r_{ij}}}^n\epsilon_k+\sum_{i'=1}^n\sum_{\substack{ j'=1\\:r_{i' j'}\le r_{ij}}}^n\epsilon_{k+1}-\alpha_i-\beta_k&\text{if }k=2,\dots,n-1,\\
\displaystyle\lambda^kc_{ij}+\sum_{i'=1}^n\sum_{\substack{ j'=1\\:r_{i' j'}\ge r_{ij}}}^n\epsilon_k-\alpha_i-\beta_k,&\text{if }k=n.\\\end{array}\right.$$

Now for each facility $j$, we define the matrix $D_j$, namely
\begin{equation}\label{matrizdj}D_j=\left(\begin{array}{c c c c}
d_{i_1j}^{1}&d_{i_1j}^{2}&\cdots&d_{i_1j}^{n}\\
d_{i_2j}^{1}&&&\\
\vdots&&\ddots&\\
d_{i_nj}^{1}&&&d_{i_nj}^{n}
\end{array}\right)
\end{equation}
where $i_1,i_2,\dots,i_n$ is a permutation of the indices $i=1,\dots,n$ which ensures $c_{i_1j}\le c_{i_2j}\le\cdots\le c_{i_nj}$.
\begin{ex}[continues=ex:firstsolution]
Next, we show the procedure that computes the elements $d_{ij}^{k}$ for all $i,k=1,\ldots,n$ of the matrix $D_1$. (\textbf{j=1})
\begin{eqnarray*}
&&d_{11}^1=\lambda^1c_{11}+r_{11}\epsilon_{2}-\alpha_1-\beta_1=4\\
&&d_{11}^2=\lambda^2c_{11}+(n^2-r_{11}+1)\epsilon_{2}+r_{11}\epsilon_{3}-\alpha_1-\beta_2=2\\
&&d_{11}^3=\lambda^3c_{11}++(n^2-r_{11}+1)\epsilon_{3}-\alpha_1-\beta_3=-9\\
&&d_{21}^1=\lambda^1c_{21}+r_{21}\epsilon_{2}-\alpha_2-\beta_1=10\\
&&d_{21}^2=\lambda^2c_{21}+(n^2-r_{21}+1)\epsilon_{2}+r_{21}\epsilon_{3}-\alpha_2-\beta_2=4\\
&&d_{21}^3=\lambda^3c_{21}++(n^2-r_{21}+1)\epsilon_{3}-\alpha_2-\beta_3=-9\\
&&d_{31}^1=\lambda^1c_{31}+r_{31}\epsilon_{2}-\alpha_3-\beta_1=24\\
&&d_{31}^2=\lambda^2c_{31}+(n^2-r_{31}+1)\epsilon_{2}+r_{21}\epsilon_{3}-\alpha_3-\beta_2=12\\
&&d_{31}^3=\lambda^3c_{31}++(n^2-r_{31}+1)\epsilon_{3}-\alpha_3-\beta_3=-4
\end{eqnarray*}
Since $r_{11}<r_{21}<r_{31}$ the  valid permutation is $(1,2,3)$. This implies that

$$D_1=\left(\begin{array}{rrr}4&2&-9\\10&4&-9\\24&12&-4\end{array}\right)\begin{array}{l}i=1\\i=2\\i=3\end{array}
$$
\end{ex}
\bigskip
We now present a dynamic programming algorithm to obtain the minimum reduced cost $\min_S\overline c_S^j$ for each $j=1,\ldots,n$.

For each couple $(i_l,k)$, we use two functions $g^j(i_l,k)$ and $S^j(i_l,k)$ representing the minimum reduced cost and the corresponding set of couples of the smaller pricing problem limited to the $l$ first rows and $k$ first columns respectively.

Our recursive procedure computes $g^j(i_l,k)$ and $S^j(i_l,k)$ for increasing values of $l$ and $k$ so that, at the end, $g^j({i_n},n)+\delta+\gamma_j=\min\limits_S\overline c_{S_j}^j$ and $S^j(i_n,n)=\text{arg}\,\min\limits_Sc_{S_j}$.

Further, the procedure exploits the following feasibility conditions on S:
\begin{enumerate}[(i)]
\item at most one couple per row and column belong to $S$.
\item if $(i_{\hat l},\hat k)$ and $(i_{\tilde l},\tilde k)\in S$ and $\hat k < \tilde k$ then $r_{i_{\hat l}j} < r_{i_{\tilde l}j}$.
\end{enumerate}

{\bf{Algorithm Pricing Subproblem}}

\begin{itemize}
\item \textbf{Step 0}

Set $g^j(i_1,1)=\min\{0,d_{i_1j}^1\}$

If $g^j(i_1,1)=d_{i_1j}^1<0$ , set $S^j(i_1,1)=\{(i_1,1)\}$. Otherwise set $S^j(i_1,1)=\emptyset$.

\item \textbf{Step 1}. For $k=2,\dots,n$.

Set $g^j(i_1,k)=\min\{d_{i_1j}^k,g^j(i_1,k-1)\}$

If $g^j(i_1,k)=g^j(i_1,k-1)$ , set $S^j(i_1,k)=S^j(i_1,k-1)$. Otherwise set $S^j(i_1,k)=\{(i_1,k)\}$.

\item \textbf{Step 2}. For $l=2,\dots,n$.

Set $g^j(i_l,1)=\min\{d_{i_lj}^1,g^j(i_{l-1},1)\}$

If $g^j(i_l,1)=g^j(i_{l-1},1)$ , set $S^j(i_l,1)=S^j(i_{l-1},k)$. Otherwise set $S^j(i_l,1)=\{(i_l,1)\}$.

\item \textbf{Step 3}. For $k,l=2,\dots,n$.

Set $g^j(i_l,k)=\min\{g^j(i_{l-1},k-1)+d_{i_lj}^k,g^j(i_{l-1},k-1),g^j(i_{l},k-1),g^j(i_{l-1},k)\}$

If $g^j(i_l,k)=g^j(i_{l-1},k-1)$ , set $S^j(i_l,k)=S^j(i_{l-1},k-1)$.

Else,  if $g^j(i_l,k)=g^j(i_{l},k-1)$ , set $S^j(i_l,k)=S^j(i_{l},k-1)$.

Else,  if $g^j(i_l,k)=g^j(i_{l-1},k)$ , set $S^j(i_l,k)=S^j(i_{l-1},k)$.

Otherwise set $S^j(i_l,k)=S^j(i_{l-1},k-1)\cup\{(i_l,k)\}$.
\end{itemize}

Obviously, if this $g^j(i_n,n)+\delta+\gamma_j$ is negative the variable $y_{S^j(i_n,n)}^j$ is a good candidate to be chosen in the next iteration of the column generation scheme.

If we solve this problem for all $j$, we get  $\overline c_{R}^j=\displaystyle \min_S\overline c_{S}^j$ and if $\overline  c_{R}^j<0$, we can activate (at least) $y_{R}^j$. Next, we solve a new reduced master problem ReLRMP with this (these) new activated variable(s).
\begin{ex}[continues=ex:firstsolution]
We show the computation of the $g^j(i_n,n)$ and  $S^j(i_n,n)$ for $j=1$.

$g^1(i_1,1)=\min\{0,4\}=0,S^1(i_1,1)=\emptyset$.

$g^1(i_1,2)=\min\{2,0\}=0,S^1(i_1,2)=\emptyset$.

$g^1(i_1,3)=\min\{-9,0\}=-9,S^1(i_1,3)=\{(1,3)\}$.

$g^1(i_2,1)=\min\{10,0\}=0,S^1(i_2,1)=\emptyset$.

$g^1(i_3,1)=\min\{24,0\}=0,S^1(i_3,1)=\emptyset$.

$g^1(i_2,2)=\min\{0+4,0,0,0\},S^1(i_2,2)=\emptyset$.

$g^1(i_3,2)=\min\{0+12,0,0,0\},S^1(i_3,2)=\emptyset$.

$g^1(i_2,3)=\min\{0-9,0,-9,0\},S^1(i_2,3)=\{(1,3)\}$.

$g^1(i_3,3)=\min\{0-4,0,-9,0\},S^1(i_3,3)=\{(1,3)\}$.

We have obtained  $g^1(i_3,3)$ and $S^1(i_3,3)=S_3$ being the potential set to be used, if the reduced cost is negative. Next, the corresponding reduced cost $\overline c_{3}^1=g^1(i_3,3)+\delta+\gamma_1=-9+0+0=-9<0$. Hence, we active variable $y_{3}^1$.

Next, the process continues with the following facilities, i.e. $j=2,3$. In this example the  optimal solution can be certified after four complete iterations of the above process.

The following table shows the objective function values and the negative reduced costs per facility obtained in each iteration.
\begin{center}
\begin{tabular}{r|c|rrr}
&&\multicolumn{3}{c}{$\displaystyle\min_Sc_S^j$}\\
&f&j=1&j=2&j=3\\
\hline
Iteration 0&12.00&-9.00&-11.00&-9.00\\
Iteration 1&12.00&-5.00&-4.00&-3.00\\
Iteration 2&12.00&-3.00&-3.00&-0.29\\
Iteration 3&9.00&0.00&0.00&0.00\\
\end{tabular}
\end{center}

\end{ex}

\subsection{Dealing with infeasibility}
One important issue when implementing a column generation procedure to solve a linear optimization problem is how to deal with infeasibility. This is specially crucial if the procedure is used within a branch-and-bound scheme to solve the linear relaxation of the problem in every node. In order to handle it, we resort to the so called Farkas pricing.

According with Farkas' Lemma, a reduced master problem is infeasible if and only if its associated dual problem is unbounded. Thus, to recover feasibility in the ReLRMP we have to revoke the certificate of unboundedness in the dual problem what can be done by adding constraints to it. Since we are only interested in recovering feasibility in ReLRMP, one can proceed in the same way that the usual pricing, but with null coefficients in the objective function of the primal. In this way, the Farkas dual problem is
$$
\begin{array}{rrll}
\max&\displaystyle\sum_{i=1}^n\alpha_i+\sum_{k=1}^n\beta_k-\sum_{j=1}^n\gamma_j-p\delta-\sum_{k=2}^nn^2\epsilon_k&&\\
s.t.&\displaystyle\sum_{\substack{i=1\\:(i,\cdot)\in S}}^n\alpha_i+\sum_{\substack{k=1\\:(\cdot,k)\in S}}^n\beta_k-\gamma_j-\delta\\
&\displaystyle-\sum_{k=2}^n\sum_{i'=1}^n\sum_{ j'=1}^n\left(\sum_{\substack{(i,k)\in S\\:r_{i' j'}\ge r_{ij}}}\epsilon_k+\sum_{\substack{(i,k-1)\in S\\:r_{i' j'}\le r_{ij}}}\epsilon_{k}\right)&\le 0&\forall\,j,S\\
&\gamma_j&\ge0&\forall\,j\\
&\delta&\ge0&\\
&\epsilon_k&\ge0&\forall\,k.\\
\end{array}
$$
We proceed to identify new variables that make the reduced master problem feasible using the dynamic programming approach replacing $c_S^j$ by zeros.

Farkas pricing is an important element in our approach because it allows to start the column generation algorithm with an empty pool of columns, although this is not advisable. Furthermore, Farkas pricing will be crucial in the branching phase to recover feasibility (whenever possible) in those nodes of the branching tree where it is lost after fixing variables.

\section{A branch-price-and-cut implementation\label{section:3}}

In this section, we precise several components of the implementation of our set partitioning formulation based on a column generation approach. \BPC\; is a branch-and-cut scheme that solves the linear relaxation of each node of the branching tree with the column generation algorithm previously described and may apply cuts to improve the obtained lower bound. (The reader is referred to \cite{Hossein2016} for another recent implementation of a \BPC.)

To calibrate the best choice of the different parameters used in our \BPC, we have performed, in all test in this section, a preliminary computational study based on  a set of 60 instances with sizes $n=20,30$ and with a time limit of 1800 sec.  Those are the smallest instances that we will eventually use in Section \ref{section:4}.

\subsection{Upper bound for the Master Problem: A GRASP heuristic and an initialization stage}\label{section:3.1}
We now present a heuristic algorithm to generate a feasible solution for $MP$. This feasible solution will provide a promising pool of initial columns as well as a good upper bound.

GRASP (\citet{Feo1989}, \citet{Feo1995}) is a well-known heuristic technique that usually exhibits good performance in short computing time. In our case, it consists in  a multistart greedy algorithm to construct a set of $p$ facilities from a randomly generated set of facilities with smaller cardinality. Following \citet{Puerto2014} we have chosen,  in a greedy manner, an initial set of $\lfloor p/2\rfloor$facilities. Next, we improve this initial solution by performing a fixed number of iterations of a local search procedure.

The greedy algorithm adds iteratively a new facility to the current set of open facilities, choosing the one with the maximum improvement of the objective value. The local search consists in an interchange heuristic between open and closed facilities.
The pseudocode of the GRASP used to solve the problem is described in Algorithm 1.

\begin{algorithm}[H]
\begin{algorithmic}[1]

\STATE Input($n,p,C,\lambda, n_1, n_2,q$);
\FOR {$n_1$ replications}
\STATE PartialSolution $\leftarrow$  ConstructRandomizedPartialSolution($q$);
\STATE Solution $\leftarrow$ ConstructGreedySolution(PartialSolution);
\FOR {$n_2$ iterations}
\STATE Solution $\leftarrow$ LocalSearch(Solution);
\STATE BestSolution $\leftarrow$ UpdateSolution(Solution, BestSolution);
\ENDFOR
\ENDFOR
\end{algorithmic}
\caption{GRASP for DOMP.\label{c3:al GRASP}}
\end{algorithm}

First of all, we would like to point out the remarkable behavior of the GRASP heuristic for this problem. In order to illustrate the appropriateness of our heuristic we have solved to optimality a number of instances of the problem   (using the MIP formulation) to be compared with those given by our GRASP. In all instances, up to a size of $n=100$, the solution provided by GRASP is always as good as the one obtained by the any of our MIP formulations with a CPU time limit of 7200 seconds, see Section \ref{section:4}.

Moreover, it is not only advisable to use the GRASP heuristic because it provides a very good upper bound thus helping the exploration of the searching tree by pruning many branches of the branch-and-bound tree, but in addition, the construction phase of the heuristic also provides a very promising pool of initial columns for the \BPC, in combination with the technique described in the following.

Since we are  solving the linear relaxation of our master problem, $LRMP$, without generating its entire set of variables, using the primal simplex algorithm, the goal of the initialization phase is to find an initial set of columns that allows solving the $MP$ by performing a small number of iterations in the column generation routine.
We create variables using a modification of the local search routine of the GRASP algorithm. Every time that we find a promising feasible solution in the heuristic, we create the variables that define that solution (CreateSetVariables(J)). Algorithm \ref{GRASP_initial} presents the pseudocode of this process.

Function CreateSetVariables(J) determines the costs involved in the solution, i.e. the minimum for each client among the open facilities. Then those costs are ordered to determine the position of each client. Once we know the couples $(i,k)$ assigned for each open facility, the corresponding variables are added to the pool.

\begin{ex}[continues=ex:firstsolution]
We illustrate the use of the function CreateSetVariables(J) with the following set $J=\{1,3\}$ (open facilities). The allocation costs for this set $J$ of open facilities are $c_{11}=1,c_{21}=3,c_{33}=1$. According to $R$, the ranks of these costs are $r_{11}=1<r_{33}=3<r_{21}=5$.  Thus, we get the couples $(1,1), (3,2)$ and $(2,3)$. This means that client $1$ goes to facility $1$ in position $1$, client $3$ goes to facility $3$ in position $2$ and client $2$ goes to facility $1$ in position $3$. Therefore, the variables $y_{\{(1,1),(2,3)\}}^1$ and $y_{\{(3,2)\}}^3$ are added to the pool.

\end{ex}

\begin{algorithm}[H]
\begin{algorithmic}[1]

\STATE Input($|J|=p$);
\STATE $\bar z= z(J)$; CreateSetVariables(J);
\FOR {$n_2$ iterations,$j_1\in J$,$j_2\in \bar J$}
\IF {$z((J\setminus\{j_1\})\cup \{j_2\})<\bar z$}
\STATE $\bar z = z((J\setminus\{j_1\})\cup \{j_2\})$; $J=(J\setminus\{j_1\})\cup \{j_2\}$;  CreateSetVariables(J);
\ENDIF
\ENDFOR
\end{algorithmic}
\caption{Initial columns.\label{GRASP_initial}}
\end{algorithm}

In order to test the helpfulness of GRASP in solving problems instances,  Table \ref{c3:tableGRASP} reports results of the 60 instances of sizes $n=20,30$ enabling or not the use of the GRASP. It shows average results of CPU time (Time(s)),   gap at termination, i.e. $100(z_{UB}-z_{LB})/z_{UB}$ (GAP(\%)), and number of unsolved problems   (in parentheses), number of nodes (\#nodes) and number of variables ($|Vars|$).

{\centering
\begin{table}[H]
\begin{center}
{\small
\begin{tabular}{rrrrr}
 \toprule
\textbf{GRASP}&\textbf{Time(s)}&{\textbf{GAP(\%)}}&\textbf{\#nodes}&\textbf{$|Vars|$}\\
 \midrule
Disabled&1350.47& -- \; (40)&33&9710\\
Enabled&1200.03&2.33(35)&19&7167\\
 \bottomrule
\end{tabular}
  \caption{CPU-Time, Number of nodes and Number of variables with and without GRASP heuristic for $n=20,30$.}
  \label{c3:tableGRASP}
}
\end{center}
\end{table}}

According with Table \ref{c3:tableGRASP}  it is  clearly advisable to use the upper bound provided by the GRASP heuristic: it reduces the number of nodes, thus improving the size of the branch-and-bound tree.

In Table \ref{t3:tableGRASP}, using the same notation that in Table \ref{c3:tableGRASP}, it is reported Time(s), $\#nodes$ and $|Vars|$ of all solved instances with sizes $n=20,30$. As one can observe from this table enabling the use of GRASP reduces the CPU time and number of nodes of the B\&B tree and at the same time reduces the overall  number of variables required by the \BPC. In addition, we would like to remark that by using the GRASP heuristic, \BPC\; is able to solve 5 more instances. Moreover, for those instances for which \BPC\; does not  certify optimality, GRASP provides an upper bound that leads to an average gap of 2.33 \%. Here, we also would like to point out that without the use of GRASP, in many cases, no feasible solutions are found within the time limit and thus, no \% gap (``--'') can be reported.

{\centering
\begin{table}[H]
\begin{center}
{\small
\begin{tabular}{cccc}
 \toprule
\textbf{GRASP}&\textbf{Time(s)}&\textbf{\#nodes}&\textbf{$|Vars|$}\\
 \midrule
Disabled&450.80&56&7664\\
Enabled&216.29&38&4500\\
 \bottomrule
\end{tabular}
  \caption{CPU-Time, Number of nodes and Number of variables  with and without GRASP heuristic for $n=20,30$. Summary of solved instances}
  \label{t3:tableGRASP}
}
\end{center}
\end{table}}

From our results, we have obtained that using GRASP heuristic one gets, on average, 4.91\% of the final number of  variables applying Algorithm \ref{GRASP_initial}.  The combination of the incumbent solution (given by GRASP) and that initial pool of variables leads to solve the considered instances faster, requiring less number of  nodes and variables to certify optimality.

Figure \ref{performanceProfilesGRASP} reports the performance profile of  GAP versus number of solved instances within a time limit of 1800 seconds, for the 60 instances with sizes $n=20,30$. The blue line reports results using GRASP and the orange one without it. It is interesting to point out that when GRASP is enabled the \BPC\;  is able to optimally solve 25 instances and the GAP of the remaining never goes beyond 10.72\%. On the other hand, if GRASP is disabled then \BPC\; only solves 20 instances but in addition, only for 4 more instances it is capable to obtain a feasible solution whereas in the remaining 36 instances the gap is greater than 100\% (no feasible solution is found).

\begin{figure}[!ht] \centering
	\begin{tikzpicture}[scale=1.0,font=\footnotesize]
	\begin{axis}[axis x line=bottom,  axis y line=left,
	xlabel=GAP,
	ylabel=Instances,
	legend style={legend pos=south east,font=\scriptsize}]

	\addplot[orange,semithick,dash pattern=on 4pt off 2pt] plot coordinates {		
(	0	,	1	)
(	0	,	2	)
(	0	,	3	)
(	0	,	4	)
(	0	,	5	)
(	0	,	6	)
(	0	,	7	)
(	0	,	8	)
(	0	,	9	)
(	0	,	10	)
(	0	,	11	)
(	0	,	12	)
(	0	,	13	)
(	0	,	14	)
(	0	,	15	)
(	0	,	16	)
(	0	,	17	)
(	0	,	18	)
(	0	,	19	)
(	0	,	20	)
(	1.27951	,	21	)
(	1.95196	,	22	)
(	6.35251	,	23	)
(	9.51121	,	24	)	
	};
	\addlegendentry{\textbf{Disabled}}

	\addplot[blue,semithick,dash pattern=on 4pt off 2pt] plot coordinates {			
(	0	,	1	)
(	0	,	2	)
(	0	,	3	)
(	0	,	4	)
(	0	,	5	)
(	0	,	6	)
(	0	,	7	)
(	0	,	8	)
(	0	,	9	)
(	0	,	10	)
(	0	,	11	)
(	0	,	12	)
(	0	,	13	)
(	0	,	14	)
(	0	,	15	)
(	0	,	16	)
(	0	,	17	)
(	0	,	18	)
(	0	,	19	)
(	0	,	20	)
(	0	,	21	)
(	0	,	22	)
(	0	,	23	)
(	0	,	24	)
(	0	,	25	)
(	0.517553	,	26	)
(	0.56926	,	27	)
(	0.643991	,	28	)
(	0.668717	,	29	)
(	0.719401	,	30	)
(	0.809135	,	31	)
(	0.875647	,	32	)
(	0.931571	,	33	)
(	1.22668	,	34	)
(	1.22713	,	35	)
(	1.26833	,	36	)
(	1.42141	,	37	)
(	1.51535	,	38	)
(	1.53571	,	39	)
(	1.5602	,	40	)
(	1.61652	,	41	)
(	1.73418	,	42	)
(	1.79682	,	43	)
(	1.91544	,	44	)
(	2.04362	,	45	)
(	2.11953	,	46	)
(	2.39634	,	47	)
(	2.47221	,	48	)
(	2.51998	,	49	)
(	2.58264	,	50	)
(	2.64914	,	51	)
(	2.65189	,	52	)
(	3.18275	,	53	)
(	3.4739	,	54	)
(	3.7324	,	55	)
(	3.7324	,	56	)
(	4.51243	,	57	)
(	5.09263	,	58	)
(	5.17248	,	59	)
(	10.7160	,	60	)

		};
		\addlegendentry{\textbf{Enabled}}
	
	\end{axis}
	\end{tikzpicture}
	\caption{Performance profile graph with GRASP enabled or disabled after 1800 seconds, GAP / \# Instances } \label{performanceProfilesGRASP}
\end{figure}
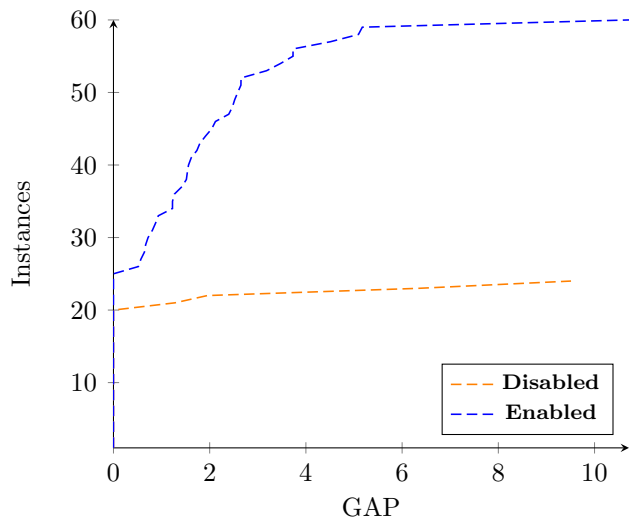

\clearpage

\subsection{Stabilization\label{section:3.2}}
When using a column generation procedure,
 the vector of dual variables may be quite different from an iteration to the next resulting in a slow convergence. For this reason, sometimes the stabilization is a critical step in order to reduce the number of variables and iterations needed to solve each reduced master problem (\citet{duMerle1999}).

In our approach, to perform the stabilization we follow the procedure in \cite{Pessoa2010} which depends on only one parameter. The idea consists in using a vector of dual variables which is a convex combination of the previous vector and the current solution of the dual problem.

Let $\pi=(\alpha, \beta, \gamma,\delta,\epsilon,\zeta)$ be a generic vector of dual multipliers, $\overline \pi$ be the best known vector of dual multipliers (found so far) and $\pi_{ReMP}$ be the current solution of the dual problem. Let $\overline c_S^j(\pi)$ be the reduced cost of $y_S^j$ computed with the dual variable $\pi$ and $LB(\pi)$ the lower bound provided by the same vector of dual multipliers, namely  $\pi$. Finally, let $z_D(\pi)$ be the value of the dual objective function of ReLRMP for the dual vector $\pi$, see (\ref{lb1}). The stabilization algorithm that we have implemented is described by the following pseudocode:
\begin{algorithm}[H]
	\begin{algorithmic}[1]
		\STATE $\Delta=\Delta_{init}$; $\overline \pi = 0$; $LB(\overline\pi)=0$; $GAP=1$;
		\WHILE {$GAP>\epsilon$}
		\STATE Solve ReLRMP, obtaining $z_{ReLRMP}$ and $\pi_{ReLRMP}$; $\pi_{st}=\Delta \pi_{ReLRMP}+(1-\Delta)\overline \pi$;
		\FOR {$j=1,\dots,n$}
		\STATE Solve the pricing using $\pi_{st}$, obtaining $S$;
		
		\STATE	\textbf{if} $\overline c_S^j (\pi_{ReLRMP})< 0$  \textbf{then} Add variable $y_S^j$; \textbf{end if}
		
		\ENDFOR
		 \STATE{$LB(\pi_{st})=z(\pi_{st}^t)+\displaystyle\sum_{\substack{S,j:y_S^j added}}\overline c_S^j(\pi_{st})$};
		\IF{At least one variable was added}
		\IF{$LB(\pi_{st})>LB(\overline\pi)$}
		\STATE $\overline\pi=\pi_{st}$;  $LB(\overline\pi)=LB(\pi_{st})$;
		\ENDIF
		\ELSE
		\STATE $\overline\pi=\pi_{st}$; $LB(\overline\pi)=LB(\pi_{st})$;
		\ENDIF
		\STATE $GAP=\frac{z_{ReLRMP}-LB(\overline\pi)}{z_{ReLRMP}}$;
		
		\STATE	\textbf{if} $GAP<1-\Delta$   \textbf{then} $\Delta=1-GAP$; \textbf{end if}
	
		\ENDWHILE
	\end{algorithmic}
	\caption{Stabilization in $ReLRMP$.\label{c3: al stabilization}}
\end{algorithm}

In words, the algorithm performs a while loop where in each iteration it makes a convex combination of the current vector of dual multipliers and the best vector of multipliers found so far. This loop ends whenever both vectors of multipliers are close enough based on the gap between the incumbent lower bound and the actual value of the reduced master problem. It is important to realize that  the coefficient (importance), $\Delta$,  given in the convex combination to $\pi_{ReLRMP}$ (the current solution of ReLRMP) increases with the number of iterations of the algorithm since $\Delta=1-GAP$ and $GAP$ decreases with the number of iterations. Eventually in the very last iterations of the stabilization algorithm we will use the actual vector of dual multipliers since $\pi_{st}\approx \pi_{ReLRMP}$.  In our implementation, we have chosen $\Delta=0.6$ based on the computational study shown in Figure \ref{performanceProfilesStabilization}. As one can observe in this figure, the best performance profile is obtained by $\Delta=0.6$ (green dashed line) because it is the configuration that solves the largest number of problem within the time limit.

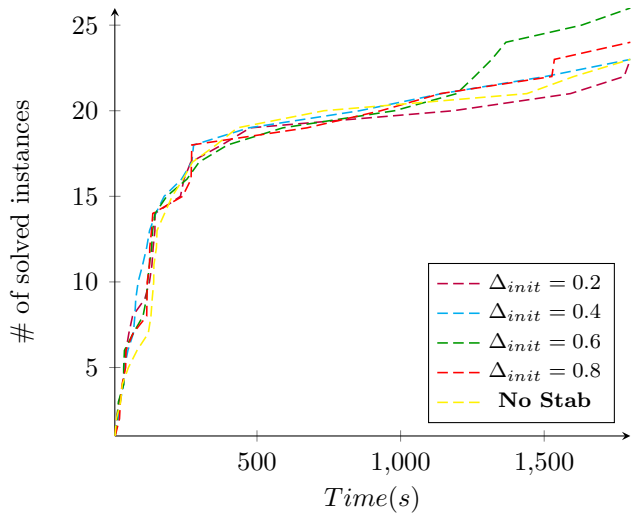
\begin{figure}[!ht] \centering
	\begin{tikzpicture}[scale=1.0,font=\footnotesize]
	\begin{axis}[axis x line=bottom,  axis y line=left,
	xlabel=$Time(s)$,
	ylabel=\# of solved instances,
	legend style={legend pos=south east,font=\scriptsize}]

	\addplot[purple,semithick,dash pattern=on 4pt off 2pt] plot coordinates {		

(	5.02	,	1	)
(	11.43	,	2	)
(	19.58	,	3	)
(	34.75	,	4	)
(	39.39	,	5	)
(	43.81	,	6	)
(	53.87	,	7	)
(	66.66	,	8	)
(	109.31	,	9	)
(	125.67	,	10	)
(	134.23	,	11	)
(	138.39	,	12	)
(	144.54	,	13	)
(	144.89	,	14	)
(	233.48	,	15	)
(	245.74	,	16	)
(	269.63	,	17	)
(	377.36	,	18	)
(	476.31	,	19	)
(	1183.53	,	20	)
(	1591.24	,	21	)
(	1778.09	,	22	)
(	1800	,	23	)

	};
	\addlegendentry{\textbf{$\Delta_{init} = 0.2$}}

	\addplot[cyan,semithick,dash pattern=on 4pt off 2pt] plot coordinates {		
	
(	5.46	,	1	)
(	16.4	,	2	)
(	23.77	,	3	)
(	35.54	,	4	)
(	38.45	,	5	)
(	50.51	,	6	)
(	69.62	,	7	)
(	74.67	,	8	)
(	77.88	,	9	)
(	86.21	,	10	)
(	101.81	,	11	)
(	115.85	,	12	)
(	125.96	,	13	)
(	147.63	,	14	)
(	176.79	,	15	)
(	237.22	,	16	)
(	272.97	,	17	)
(	279.33	,	18	)
(	471.53	,	19	)
(	857.19	,	20	)
(	1142.57	,	21	)
(	1504.01	,	22	)
(	1800	,	23	)

	};
	\addlegendentry{\textbf{$\Delta_{init} = 0.4$}}	

	\addplot[Darkgreen,semithick,dash pattern=on 4pt off 2pt] plot coordinates {		
	
(	6.44	,	1	)
(	12.32	,	2	)
(	17.57	,	3	)
(	35.61	,	4	)
(	35.84	,	5	)
(	38.45	,	6	)
(	68.84	,	7	)
(	103.17	,	8	)
(	114.72	,	9	)
(	119.24	,	10	)
(	125.26	,	11	)
(	133.1	,	12	)
(	136.5	,	13	)
(	145.14	,	14	)
(	185.61	,	15	)
(	255.92	,	16	)
(	298.51	,	17	)
(	397.49	,	18	)
(	589.71	,	19	)
(	979.08	,	20	)
(	1195.07	,	21	)
(	1257.66	,	22	)
(	1317.39	,	23	)
(	1366.6	,	24	)
(	1630.31	,	25	)
(	1800	,	26	)
	
	};
	\addlegendentry{\textbf{$\Delta_{init} = 0.6$}}	

	\addplot[red,semithick,dash pattern=on 4pt off 2pt] plot coordinates {		
	
(	6.28	,	1	)
(	21.27	,	2	)
(	24.26	,	3	)
(	30.09	,	4	)
(	42.61	,	5	)
(	43.9	,	6	)
(	68.5	,	7	)
(	115.97	,	8	)
(	117.01	,	9	)
(	117.86	,	10	)
(	125.81	,	11	)
(	126.57	,	12	)
(	132.86	,	13	)
(	137.22	,	14	)
(	240.89	,	15	)
(	270.46	,	16	)
(	271.96	,	17	)
(	272.85	,	18	)
(	673.87	,	19	)
(	931.07	,	20	)
(	1138.52	,	21	)
(	1526.73	,	22	)
(	1535.45	,	23	)
(	1800	,	24	)

	};
	\addlegendentry{\textbf{$\Delta_{init} = 0.8$}}	

	\addplot[yellow,semithick,dash pattern=on 4pt off 2pt] plot coordinates {		
	
(	5.78	,	1	)
(	10.35	,	2	)
(	25.26	,	3	)
(	30.95	,	4	)
(	52.35	,	5	)
(	82.44	,	6	)
(	121.51	,	7	)
(	130.81	,	8	)
(	135.14	,	9	)
(	140.37	,	10	)
(	141	,	11	)
(	149.84	,	12	)
(	151.91	,	13	)
(	178.24	,	14	)
(	204.6	,	15	)
(	242.35	,	16	)
(	275.54	,	17	)
(	357.48	,	18	)
(	431.54	,	19	)
(	733.01	,	20	)
(	1439.63	,	21	)
(	1603.3	,	22	)
(	1800	,	23	)

	};
	\addlegendentry{\textbf{No Stab}}

	\end{axis}
	\end{tikzpicture}
	\caption{Performance profile graph with different combination of $\Delta_{init}$, \#solved instances / $n$ } \label{performanceProfilesStabilization}
\end{figure}

In order to show the performance of the stabilization algorithm (Algorithm \ref{c3: al stabilization}), we report in Figure \ref{fig:dd} the evolution of the lower and upper bounds with respect to number of iterations.
Results reported here correspond to a single example. When Stabilization generally results in a better behavior. One can realize that the dual bound is not infinity at iteration 0 and that it does not improve for some iterations. The reason is because we start with a feasible solution of the problem.
\begin{figure}
\hspace*{-0.75cm}
\begin{subfigure}{.6\textwidth}
  \centering
  \includegraphics[scale=0.7]{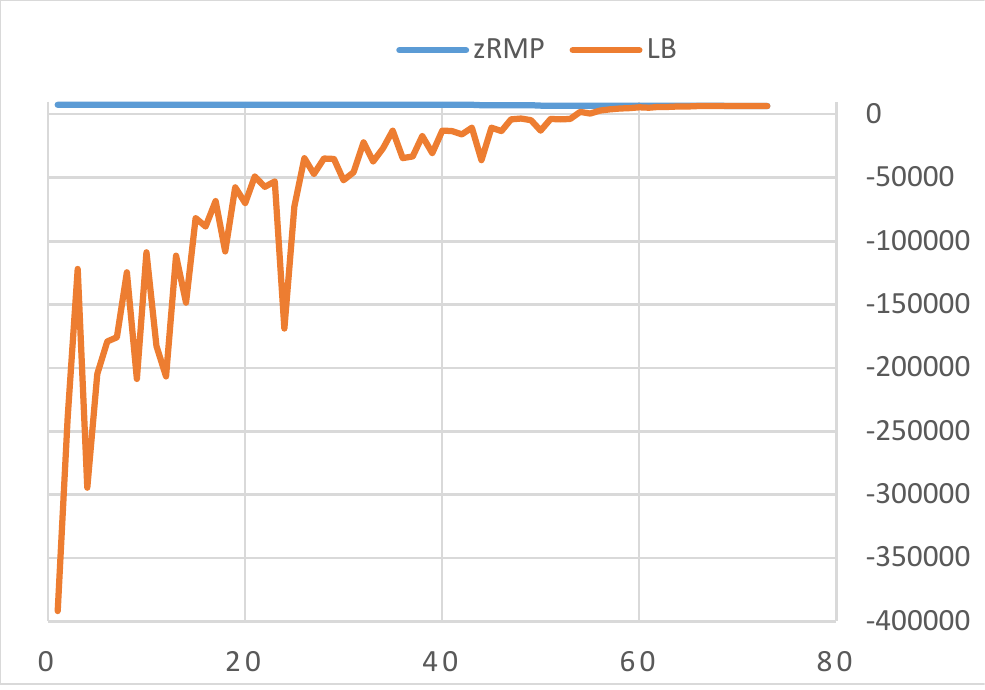}
  \caption{Stabilization disabled}
  \label{fig:sfig1}
\end{subfigure}%
\hspace*{-1.25cm}
\begin{subfigure}{.6\textwidth}
  \centering
  \includegraphics[scale=0.7]{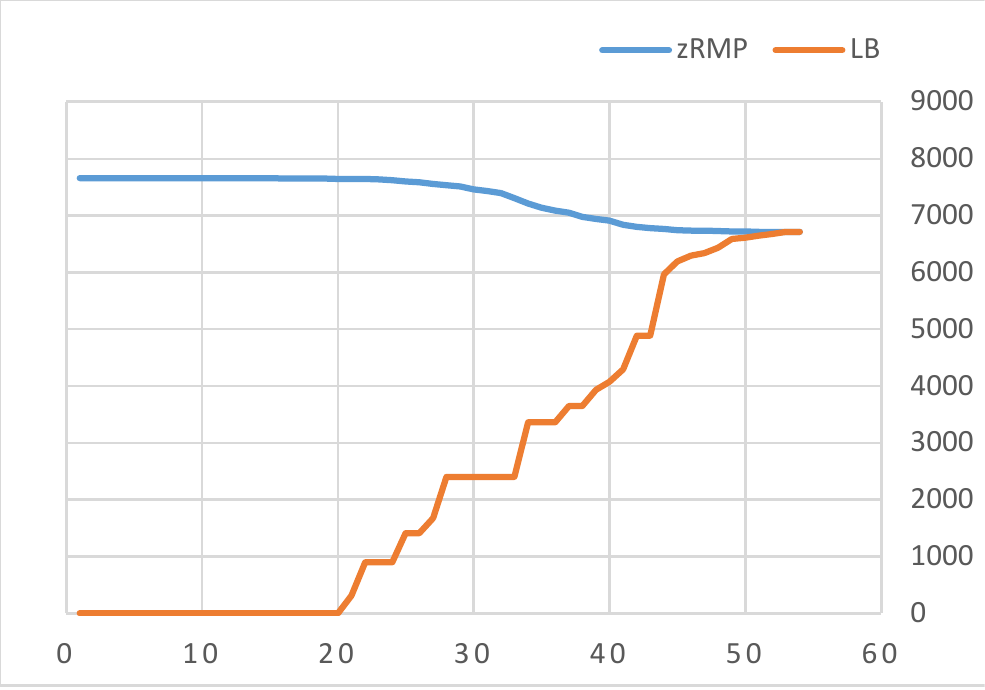}
  \caption{Stabilization enabled}
  \label{fig:sfig2}
\end{subfigure}
\caption{Bound's behavior at the root node in a particular instance on successive iterations.}
\label{fig:dd}
\end{figure}

The control over the dual variables significantly improves the necessary number of iterations and the number of variables used to certify optimality. Note that this improvement becomes more important where $MP$ is solved using a branch-and-bound procedure because the number of variables should be small in every node.


\subsection{HurryPricer: the Pricer heuristic\label{section:3.3}}

The pricing subproblem can be solved optimally by the dynamic programming algorithm described in Section \ref{section:2.5} with a worst case complexity of $O(n^3)$. However, this complexity may be excessive if the number of calls to that routine is large. For that reason, we have developed an alternative pricer heuristic that looks, in a greedy manner, for new variables in the pricing process with much less computational burden. Of course, if the heuristic does not find any variable to be added we need to resort to the exact pricer either to certify optimality or to find alternative variables that were not found in the heuristic phase.

A brief pseudocode description of the heuristic pricer is given in the Appendix.

\begin{algorithm}
	\caption{HurryPricer}
	\begin{algorithmic}[1]
		\STATE Input($\alpha, \beta, \gamma, \epsilon,\delta, \zeta$);  $S = \emptyset$;
		\FOR{a set of selected $j$}
		\STATE $\bar{c}^j_S = 0$; $k'= 0$; $l = 1$;
		\WHILE{($k' \neq n$) and ($l < n + 1 $)}
		\STATE Continue = True;  $k = k' + 1$;
		\WHILE{(Continue is True) and ($k < n +1$)}
		\STATE $d_{i_lj}^k = \lambda^k c_{i_lj} + r_{i_lj}  	\epsilon_k
		+  (n^2 - r_{i_lj} + 1 ) \epsilon_{k-1} - \alpha_{i_l} - \beta_k $;
		\IF{$d_{i_lj}^k < 0$}
		\IF {we consider cuts}
		\STATE $d_{i_lj}^k = d_{i_lj}^k + \sum^n_{ i'=1} \sum^n_{\substack{ i'{j}=1: \\ r_{ i' j'} \leq r_{i_lj}}} \zeta_{ i' j'}^k + \sum^n_{ i'=1} \sum^n_{\substack{ j'=1: \\ r_{ i' j'} \geq r_{i_lj}}}  \zeta_{ i' j'}^{k-1}$;
		\IF{$d_{i_lj}^k < 0$}
		\STATE 	$\bar{c}^j_S = \bar{c}^j_S + d_{i_lj}^k$;  $S_j = S_j \cup \{ (i_l,j) \}$;
		\STATE Continue = False; $k' = k$;
		\ENDIF
		\ELSE
		\STATE 	$\bar{c}^j_S = \bar{c}^j_S + d_{i_lj}^k$;  $S_j = S_j \cup \{ (i_l,j) \}$;
		
		\STATE Continue = False; $k' = k$;				
		\ENDIF	
		\ENDIF
		\STATE $k = k + 1$;
		\ENDWHILE
		\STATE $l = l + 1$;	
		\ENDWHILE
		\IF{$\bar{c}^j_S + \delta + \gamma_j < 0$}
		\STATE 	$S = S \cup S_j$;
		\ENDIF
		\ENDFOR
		\RETURN $S$;
	\end{algorithmic}\label{hurryPricer}
\end{algorithm}

In the following we analyze whether is is advisable to combine stabilization techniques and pricing heuristics in the pricing subproblem. We show in Figure \ref{performanceProfilesHP}  the performance profiles of time versus number of solved instances. From this figure one can observe that combining stabilization and Hurry Pricer seems to have a slightly better behavior than the remaining options.  This conclusion is reinforced by the data shown in Table \ref{table:stabHP}  based on computing time, number of variables and nodes required by the different combinations.

\begin{figure}[!ht] \centering
	\begin{tikzpicture}[scale=1.0,font=\footnotesize]
	\begin{axis}[axis x line=bottom,  axis y line=left,
	xlabel=$Time(s)$,
	ylabel=\# of solved instances,
	legend style={legend pos=south east,font=\scriptsize}]
	
	\addplot[Darkgreen,semithick,dash pattern=on 4pt off 2pt] plot coordinates {		
	
(	2.1	,	1	)
(	5.25	,	2	)
(	5.34	,	3	)
(	11.08	,	4	)
(	11.87	,	5	)
(	15.23	,	6	)
(	40.75	,	7	)
(	41.57	,	8	)
(	68.75	,	9	)
(	102.15	,	10	)
(	108.59	,	11	)
(	113.08	,	12	)
(	116.97	,	13	)
(	123.49	,	14	)
(	123.83	,	15	)
(	153.41	,	16	)
(	162.46	,	17	)
(	265.99	,	18	)
(	312.25	,	19	)
(	438	,	20	)
(	880.77	,	21	)
(	1022.28	,	22	)
(	1240.06	,	23	)
(	1316.86	,	24	)
(	1661.71	,	25	)
(	1800	,	26	)

	};
	\addlegendentry{\textbf{Stab+HP}}	

	\addplot[blue,semithick,dash pattern=on 4pt off 2pt] plot coordinates {		
	
(	4.04	,	1	)
(	5.61	,	2	)
(	6.38	,	3	)
(	20.73	,	4	)
(	25.77	,	5	)
(	40.99	,	6	)
(	46.7	,	7	)
(	66.99	,	8	)
(	91.68	,	9	)
(	100.19	,	10	)
(	104.69	,	11	)
(	113.2	,	12	)
(	136.29	,	13	)
(	142.39	,	14	)
(	149.53	,	15	)
(	161.86	,	16	)
(	219.38	,	17	)
(	301.17	,	18	)
(	416.76	,	19	)
(	436.35	,	20	)
(	712.1	,	21	)
(	1035.18	,	22	)
(	1419.84	,	23	)
(	1582.06	,	24	)
(	1620.33	,	25	)
(	1800	,	26	)

	};
	\addlegendentry{\textbf{HP}}	

	\addplot[red,semithick,dash pattern=on 4pt off 2pt] plot coordinates {		

(	6.44	,	1	)
(	12.32	,	2	)
(	17.57	,	3	)
(	35.61	,	4	)
(	35.84	,	5	)
(	38.45	,	6	)
(	68.84	,	7	)
(	103.17	,	8	)
(	114.72	,	9	)
(	119.24	,	10	)
(	125.26	,	11	)
(	133.1	,	12	)
(	136.5	,	13	)
(	145.14	,	14	)
(	185.61	,	15	)
(	255.92	,	16	)
(	298.51	,	17	)
(	397.49	,	18	)
(	589.71	,	19	)
(	979.08	,	20	)
(	1195.07	,	21	)
(	1257.66	,	22	)
(	1317.39	,	23	)
(	1366.6	,	24	)
(	1630.31	,	25	)
(	1800	,	26	)

	};
	\addlegendentry{\textbf{Stab}}	

	\addplot[yellow,semithick,dash pattern=on 4pt off 2pt] plot coordinates {	

(	5.78	,	1	)
(	10.35	,	2	)
(	25.26	,	3	)
(	30.95	,	4	)
(	52.35	,	5	)
(	82.44	,	6	)
(	121.51	,	7	)
(	130.81	,	8	)
(	135.14	,	9	)
(	140.37	,	10	)
(	141	,	11	)
(	149.84	,	12	)
(	151.91	,	13	)
(	178.24	,	14	)
(	204.6	,	15	)
(	242.35	,	16	)
(	275.54	,	17	)
(	357.48	,	18	)
(	431.54	,	19	)
(	733.01	,	20	)
(	1439.63	,	21	)
(	1603.3	,	22	)
(	1800	,	23	)

	};
	\addlegendentry{\textbf{No Stab No HP}}

	\end{axis}
	\end{tikzpicture}
	\caption{Performance profile graph of \#solved instances with different combinations of Hurry Pricer (HP) and stabilization (Stab). } \label{performanceProfilesHP}
\end{figure}
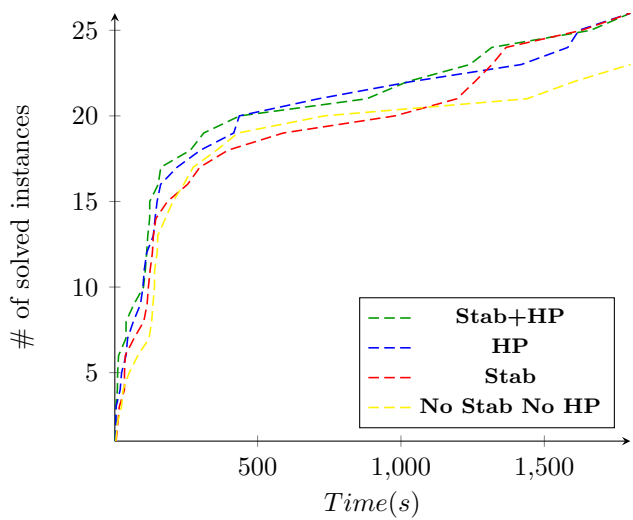

{\centering
\begin{table}[H]
\begin{center}
{\small
\begin{tabular}{crrrr}
 \toprule
HP&  Stab & Time (s) & Variables & Nodes\\
 \midrule
  No&Yes&422.62&6023&38\\
  Yes&No&358.41&5437&37\\
  Yes&Yes&333.75&5128&33\\
 \bottomrule
\end{tabular}
  \caption{Average CPU-Time, number of variables and number of nodes with different strategies of stabilization for the 25 solved instances in 1800 seconds.}
  \label{table:stabHP}
}
\end{center}
\end{table}}


\subsection{Preprocessing\label{section:3.4}}
In order to improve the performance of the algorithm we use two different preprocessings to set some variables to zero. Our approach is based on Claims 1 and 2 in \cite{Labbe2017}. The reader may observe that although those results fix to zero $x_{ij}^k$ variables, this variable-fixing can be translated to the new setting by the relation $x_{ij}^k=\sum_{S\ni(i,k)}y_S^j$ between the variables in $WOC$ and $MP$ formulations.

Therefore, the above results imply that those variables $y_S^j$ such that $(i,k)\in S$ and $x_{ij}^k=0$ will not be considered to be added to the ReLRMP. This can be simply enforced by setting the corresponding $d_{ij}^k=0$ in every pricing subproblem.

{\subsection{Branching strategies\label{section:3.5}}

Branching on original variables is a common option on Mixed Integer Master Problems where some set partition constraints are involved. See for instance \citet{Johnson1989}. In spite of that, we have also considered other branching strategies as using the set partitioning variables or the Ryan and Foster branching, \citet{Ryan1981,Barnhart1998}. However, these two alternatives were discarded because branching in original variables our pricing subproblem is polynomially solvable whereas using any of the other branching strategies mentioned above, makes it NP-hard.

Recall that $x_{ij}^k=\sum_{S\ni(i,k)}y_S^j$, thus, a way to branch on a fractional solution can be derived directly from satisfying integrality conditions of original variables.
\begin{pro}
If $x_{ij}^k\in\{0,1\}$ for $i,j,k=1,\dots, n$, then $y_S^j\in\{0,1\}$.
\begin{proof}
Suppose on the contrary there exists a variable with fractional value $y_{S^\prime}^{j^\prime}$. Since $x_{ij}^k$ are binary for all $i,j,k$ (in particular for $i_1,j^{\prime},k_1$ where $(i_1,k_1)$ a pair of $S^\prime$), there must be another fractional variable $y_{S^{\prime\prime}}^{j^\prime}$ such that $(i_1,k_1)\in S^{\prime\prime}$.

Note that $S^{\prime\prime}\neq S^{\prime}$ since the column generation procedure never generates duplicate variables, there is a pair $(i_2,k_2)$ such that either $(i_2,k_2) \in S^{\prime}$ or $(i_2,k_2)\in S^{\prime\prime}$ but not both.
Therefore,  we obtain the following relationship
$$1\ge\sum_{S\ni(i_1,k_1)}y_S^{j^\prime}>\sum_{S\ni(i_2,k_2)}y_S^{j^\prime}>0.$$
The first inequality comes directly from the formulation. The second inequality is strict because the term $\sum_{S\ni(i_2,k_2)}y_S^{j^\prime}$ has at least one fractional variable less than the term $\sum_{S\ni(i_1,k_1)}y_S^{j^\prime}$. The third inequality is strict because of the choice of $(i_2,k_2)$. Finally, a contradiction is found because $x_{i_2k_2}^{j\prime}$ is not binary.
\end{proof}
\end{pro}
The reader may note that this branching can be seen as a SOS1 branching \citep{Beale1970} since at most one of the above $y_S^j$ variables can assume the value 1.

The way to implement this branching in the pricing subproblem is to set locally (in the current node) to zero the $y_S^j$ variables which are in conflict with the condition implied by the branch $x_{ij}^k=0$ or $x_{ij}^k=1$.

In the case $x_{ij}^k=0$ we set $y_S^j=0$ for all sets $S$ containing couples $(i,k)\in S$.
Analogously, in the case $x_{ij}^k=1$ we set $y_S^{j'}=0$ for all sets $S$ containing  $(i,k)\in S$ such that $j\ne j'$, $(i',k)\in S$ such that $i\ne i'$  or $(i,k')\in S$ such that $k\ne k'$.

This condition can be transferred to the pricing subproblem modifying the $d_{ij}^k$ coefficients accordingly. Specifically, this transformation is done as follows:
\begin{itemize}
\item If $x_{ij}^k=0$ then $ d_{ij}^k=0.$
\item If $x_{ij}^k=1 $ then $
\left\{\begin{array}{l}
d_{i j'}^k=0,\quad \forall j'\neq j.\\
d_{i' j'}^k=0,\quad \forall j'\forall i'\neq i.\\
d_{i j'}^{ k'}=0,\quad \forall j'\forall k'\neq k.\\
\end{array}\right.$
\end{itemize}
Moreover,  it is also well-known that branching on SOS constraints (original variables) gives rise to more balanced branching trees (see e.g. Chapter 7 of \citep{Wolsey1998}) than branching on the variables of $MP$.

Among the fractional original variables one has to decide which will be the next variable to branch on. One of the easiest techniques for this choice is to consider the \emph{most fractional variable}. This is not difficult to implement but it is not better than choosing randomly (\citep{Achterberg2005}). Alternative techniques are \emph{pseudocost branching} (\citep{Benichou1971}) or \emph{strong branching} (\citet{Applegate1995}) although they are rather costly.

This issue has motivated us to propose another rule to select the variable to branch on, based on the improvement of the bounds in each of the new created nodes. We use the following indices corresponding to the down and up branches of the variable $x_{ij}^k$:
\begin{equation}
\varsigma_{ij}^{k,-}=\frac{\lambda^kc_{ij}}{x_{ij}^k}\text{ and }\varsigma_{ij}^{k,+}=\frac{\lambda^kc_{ij}}{1-x_{ij}^k}. \label{branch-n1}
\end{equation}
They account, respectively, for the unitary contribution to the objective function due to fixing the variable $x_{ij}^k$ either to zero (down branching) or to one (up branching). Branching down stimulates the improvement of the lower bound, whereas branching up helps the problem to find integer solutions.

We have tested several strategies that make use of the indices, $\varsigma$, defined above.
\begin{description}
\item[ Strategy 1: ] $\arg\min\{\theta\varsigma_{ij}^{k,-}+(1-\theta)\varsigma_{ij}^{k,+}:0<x_{ij}^k<1\}$
\item[ Strategy 2: ] $\arg\min\{\min\{\varsigma_{ij}^{k,-},\varsigma_{ij}^{k,+}\}:0<x_{ij}^k<1\}$
\item[ Strategy 3: ] $\arg\min\{\max\{\varsigma_{ij}^{k,-},\varsigma_{ij}^{k,+}\}:0<x_{ij}^k<1\}$.
\end{description}

Based on our computational experience (see Figure \ref{performanceProfilesBranching}), we have concluded that the best strategy to choose the following variable to branch on corresponds to strategy 1 with $\theta=0.5$.

\begin{figure}[!ht] \centering
	\begin{tikzpicture}[scale=1.0,font=\footnotesize]
	\begin{axis}[axis x line=bottom,  axis y line=left,
	xlabel=$Time(s)$,
	ylabel=\# of solved instances,
	legend style={legend pos=south east,font=\scriptsize}]

	\addplot[purple,semithick,dash pattern=on 4pt off 2pt] plot coordinates {		

(	2.21	,	1	)
(	5.23	,	2	)
(	5.29	,	3	)
(	10.97	,	4	)
(	11.45	,	5	)
(	39.67	,	6	)
(	40.67	,	7	)
(	82.98	,	8	)
(	95.05	,	9	)
(	109.86	,	10	)
(	115.26	,	11	)
(	120.01	,	12	)
(	149.42	,	13	)
(	230.57	,	14	)
(	257.13	,	15	)
(	334.43	,	16	)
(	401.19	,	17	)
(	588.17	,	18	)
(	666.21	,	19	)
(	1595.09	,	20	)
(	1779.86	,	21	)
(	1800	,	22	)

	};
	\addlegendentry{\textbf{S1, $\theta = 0.0$}}

	\addplot[cyan,semithick,dash pattern=on 4pt off 2pt] plot coordinates {		
	
(	2.17	,	1	)
(	5.12	,	2	)
(	5.32	,	3	)
(	10.95	,	4	)
(	11.39	,	5	)
(	33.14	,	6	)
(	39.93	,	7	)
(	40.87	,	8	)
(	52.11	,	9	)
(	84.04	,	10	)
(	104.39	,	11	)
(	109.78	,	12	)
(	120.44	,	13	)
(	187.92	,	14	)
(	243.08	,	15	)
(	257.25	,	16	)
(	260.34	,	17	)
(	327.23	,	18	)
(	546.82	,	19	)
(	561.19	,	20	)
(	675.92	,	21	)
(	884.24	,	22	)
(	1040.66	,	23	)
(	1508.16	,	24	)
(	1749.27	,	25	)
(	1779.21	,	26	)
(	1800	,	27	)

	};
	\addlegendentry{\textbf{S1, $\theta = 0.1$}}	

	\addplot[red,semithick,dash pattern=on 4pt off 2pt] plot coordinates {		
	
(	2.15	,	1	)
(	5.06	,	2	)
(	5.34	,	3	)
(	11.09	,	4	)
(	11.53	,	5	)
(	23.59	,	6	)
(	39.84	,	7	)
(	40.94	,	8	)
(	51.77	,	9	)
(	70.34	,	10	)
(	99.7	,	11	)
(	102.46	,	12	)
(	110.15	,	13	)
(	114.63	,	14	)
(	119.92	,	15	)
(	143.5	,	16	)
(	155.49	,	17	)
(	225.08	,	18	)
(	230.99	,	19	)
(	256.81	,	20	)
(	646.63	,	21	)
(	765.02	,	22	)
(	884.63	,	23	)
(	988.98	,	24	)
(	1546.69	,	25	)
(	1800	,	26	)

	};
	\addlegendentry{\textbf{S1, $\theta = 0.3$}}	

	\addplot[Darkgreen,semithick,dash pattern=on 4pt off 2pt] plot coordinates {		
	
(	2.06	,	1	)
(	5.11	,	2	)
(	5.29	,	3	)
(	10.94	,	4	)
(	11.4	,	5	)
(	19.68	,	6	)
(	40.14	,	7	)
(	40.9	,	8	)
(	55.84	,	9	)
(	84.16	,	10	)
(	99.7	,	11	)
(	102.1	,	12	)
(	110.86	,	13	)
(	113.95	,	14	)
(	120.5	,	15	)
(	126.15	,	16	)
(	193.67	,	17	)
(	243.37	,	18	)
(	256.85	,	19	)
(	273.74	,	20	)
(	509.03	,	21	)
(	882.16	,	22	)
(	1076.95	,	23	)
(	1334.4	,	24	)
(	1346.21	,	25	)
(	1744.06	,	26	)
(	1800	,	27	)

	};
	\addlegendentry{\textbf{S1, $\theta = 0.5$}}	

	\addplot[yellow,semithick,dash pattern=on 4pt off 2pt] plot coordinates {		
	
(	2.09	,	1	)
(	5.09	,	2	)
(	5.36	,	3	)
(	10.86	,	4	)
(	11.32	,	5	)
(	17.04	,	6	)
(	40.07	,	7	)
(	40.85	,	8	)
(	55.73	,	9	)
(	99.85	,	10	)
(	102.5	,	11	)
(	105.61	,	12	)
(	110.49	,	13	)
(	114.23	,	14	)
(	118.71	,	15	)
(	128.71	,	16	)
(	194.18	,	17	)
(	194.48	,	18	)
(	257.99	,	19	)
(	415.53	,	20	)
(	829.63	,	21	)
(	1482.29	,	22	)
(	1534.08	,	23	)
(	1800	,	24	)

	};
	\addlegendentry{\textbf{S1, $\theta = 0.7$}}	

	\addplot[pink,semithick,dash pattern=on 4pt off 2pt] plot coordinates {		
	
(	2.13	,	1	)
(	5.18	,	2	)
(	5.19	,	3	)
(	11	,	4	)
(	11.45	,	5	)
(	16.25	,	6	)
(	39.75	,	7	)
(	40.89	,	8	)
(	67.75	,	9	)
(	99.93	,	10	)
(	105.91	,	11	)
(	110.23	,	12	)
(	113.74	,	13	)
(	121.18	,	14	)
(	131.08	,	15	)
(	159.55	,	16	)
(	161.29	,	17	)
(	241.72	,	18	)
(	257.49	,	19	)
(	398.22	,	20	)
(	924.52	,	21	)
(	1276.27	,	22	)
(	1477.08	,	23	)
(	1497.38	,	24	)
(	1622.41	,	25	)
(	1800	,	26	)

	};
	\addlegendentry{\textbf{S1, $\theta = 0.9$}}

	\addplot[orange,semithick,dash pattern=on 4pt off 2pt] plot coordinates {		
	
(	2.1	,	1	)
(	5.25	,	2	)
(	5.34	,	3	)
(	11.08	,	4	)
(	11.87	,	5	)
(	15.23	,	6	)
(	40.75	,	7	)
(	41.57	,	8	)
(	68.75	,	9	)
(	102.15	,	10	)
(	108.59	,	11	)
(	113.08	,	12	)
(	116.97	,	13	)
(	123.49	,	14	)
(	123.83	,	15	)
(	153.41	,	16	)
(	162.46	,	17	)
(	265.99	,	18	)
(	312.25	,	19	)
(	438	,	20	)
(	880.77	,	21	)
(	1022.28	,	22	)
(	1240.06	,	23	)
(	1316.86	,	24	)
(	1661.71	,	25	)
(	1800	,	26	)

	};
	\addlegendentry{\textbf{S1, $\theta = 1.0$}}

	\addplot[blue,semithick,dash pattern=on 4pt off 2pt] plot coordinates {		
	
(	2.25	,	1	)
(	5.14	,	2	)
(	5.33	,	3	)
(	11.12	,	4	)
(	11.49	,	5	)
(	39.9	,	6	)
(	40.64	,	7	)
(	41.09	,	8	)
(	95.09	,	9	)
(	111.06	,	10	)
(	115.08	,	11	)
(	121.72	,	12	)
(	142.28	,	13	)
(	162.71	,	14	)
(	187.37	,	15	)
(	257.44	,	16	)
(	273.07	,	17	)
(	418.7	,	18	)
(	627.78	,	19	)
(	1061.43	,	20	)
(	1598.21	,	21	)
(	1762.61	,	22	)
(	1800	,	23	)

	};
	\addlegendentry{\textbf{S2}}

	\addplot[brown,semithick,dash pattern=on 4pt off 2pt] plot coordinates {		
	
(	2.1	,	1	)
(	5.18	,	2	)
(	5.27	,	3	)
(	10.65	,	4	)
(	11.91	,	5	)
(	17.59	,	6	)
(	39.71	,	7	)
(	40.81	,	8	)
(	65.43	,	9	)
(	99.82	,	10	)
(	102.2	,	11	)
(	110.01	,	12	)
(	110.17	,	13	)
(	115.04	,	14	)
(	118.38	,	15	)
(	144.96	,	16	)
(	230.43	,	17	)
(	255.15	,	18	)
(	291.97	,	19	)
(	503.6	,	20	)
(	507.55	,	21	)
(	1378.24	,	22	)
(	1656.89	,	23	)
(	1737.4	,	24	)
(	1761.19	,	25	)
(	1800	,	26	)

	};
	\addlegendentry{\textbf{S3}}		
	
	\end{axis}
	\end{tikzpicture}
	\caption{Performance profile graph of \#solved instances using different branching strategies.} \label{performanceProfilesBranching}
\end{figure}
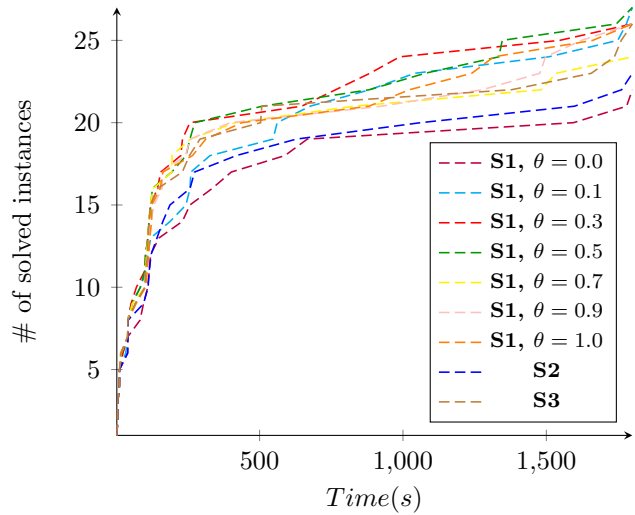


Each node of the branching tree can be fathomed before it is fully processed comparing lower bounds, as given by (\ref{lb1}) and (\ref{lb2}), with the current incumbent solution. This strategy implies reducing the number of calls to the pricing subproblem and as a result savings in the number of variables added to the restricted master problem.

\subsection{Valid inequalities\label{section:3.6}}

Clearly, the addition of valid inequalities (\ref{cuts}) to $MP$ modifies the structure of the master problem and thus the pricing must be modified accordingly. Let us denote by $\zeta_{ij}^k$ the dual variable associated with valid inequality (\ref{cuts}) for indices $i,j,k$. After some calculation, one obtains the following expression of the reduced costs of variable $y_S^j$:
\scriptsize
$$\overline c_S^j=c_S^j+\gamma_j^*+\delta^*+\sum_{k=2}^n\sum_{i'=1}^n\sum_{ j'=1}^n\left(\sum_{\substack{(i,k)\in S\\:r_{i' j'}\ge r_{ij}}}(\epsilon_k^*+\zeta_{i' j'}^{k*})+\sum_{\substack{(i,k-1)\in S\\:r_{i' j'}\le r_{ij}}}(\epsilon_{k}^*+\zeta_{i' j'}^{k*})\right)-\sum_{\substack{i=1\\:(i,\cdot)\in S}}^n\alpha_i^*-\sum_{\substack{k=1\\:(\cdot,k)\in S}}^n\beta_k^*.$$
\normalsize
Furthermore, solving the pricing subproblem to find a new column or to certify optimality of the column generation algorithm requires to adapt the dynamic programming algorithm that computes the $g(i_l,k)$ terms using the new dual multipliers. This implies to modify the $D_j$ matrices. Once again, after some calculations the modified $d_{ij}^k$ elements are now given by:
\scriptsize
$$d_{ij}^k=\left\{\begin{array}{ll}\displaystyle\lambda^kc_{ij}+\sum_{i'=1}^n\sum_{\substack{ j'=1\\:r_{i' j'}\le r_{ij}}}(\epsilon_{k+1}+\zeta_{i' j'}^{k+1})-\alpha_i-\beta_k&\text{if }k=1\\
\displaystyle\lambda^kc_{ij}+\sum_{i'=1}^n\sum_{\substack{ j'=1\\:r_{i' j'}\ge r_{ij}}}^n(\epsilon_k+\zeta_{i' j'}^{k})+\sum_{i'=1}^n\sum_{\substack{ j'=1\\:r_{i' j'}\le r_{ij}}}(\epsilon_{k+1}+\zeta_{i' j'}^{k+1})-\alpha_i-\beta_k&\text{if }k=2,\dots,n-1\\
\displaystyle\lambda^kc_{ij}+\sum_{i'=1}^n\sum_{\substack{ j'=1\\:r_{i' j'}\ge r_{ij}}}^n(\epsilon_k+\zeta_{i' j'}^{k})-\alpha_i-\beta_k,&\text{if }k=n.\\\end{array}\right.$$
\normalsize
These new elements allow us to apply the adapted column generation algorithm to solve LRMP, reinforced with valid inequalities (\ref{cuts}). The implementation details of how to adapt these new elements within the pricer and the hurry pricer can be found in the appendix \ref{ap:hp}.

To justify the use of the mentioned cuts we have done some preliminary computational experiments with instances of sizes $n=50$ and $60$.  Table \ref{ResultsCuts} compares the behavior of the standard branch-and-price without cuts, ($\textbf{B\&P}(MP)$), against the strategy with cuts, $\textbf{B\&P\&C}(MP)$.

\begin{table}[!ht]
	\begin{center}
	  \scriptsize
		\begin{tabular}{ll|rrr|rrr}
			&	& \multicolumn{3}{c|}{$n=50$} & \multicolumn{3}{c}{$n=60$}  \\
			&	& $p=12$ & $p=16$ & $p=25$& $p=15$ & $p=20$ & $p=30$ \\ \hline				 \textbf{B\&P} &   $Time(s)$&7200.00&7200.00&7200.00&7200.00&7200.00&7200.00  \\
						
			(MP) & $|Vars|$&30277&24410&16617&28443&24146&19996 \\
			
			& $|Nodes|$&1016&2728&6149&1091&2013&3736 \\
			
                   & $\# unsolved$&10&10&10&10&10&10  \\

                   & $Gap(\%)$ &6.44&7.60&9.45&8.20&8.83&11.59 \\\hline

			\textbf{\BPC} & $Time(s)$&7200.00&7200.00&6697.44&7200.00&6864.94&7200.00\\	
				
			(MP) & $|Vars|$&14971&13627 &13725 &21094&16077&17634 \\
			
			& $|Nodes|$&55&1 &1&599&535&512  \\
			
			& $|Cuts|$&7807&7907 &9183&12999&16061&13342  \\

                   & $\# unsolved$&10&10 &9 &10&9&10 \\

                   & $Gap(\%)$&3.96&5.06 &3.87&7.04&6.83&7.48  \\\hline
				
		\end{tabular}
		\caption{Numerical results with and without cuts } \label{ResultsCuts}
	\end{center}
\end{table}

From Table \ref{ResultsCuts}, we conclude that it is always better to add cuts because the final gap is always smaller with this strategy. This solution scheme has been implemented and the results are reported in the next section.

\section{Computational Experiments\label{section:4}}

The \BPC\; implementation of the formulation $MP$  has been experimentally compared with the B\&C implementation of the formulation $WOC$  on the instances detailed below. The \BPC \; algorithm considered in these experiments is based on the description in the previous section.

The computer used for these tests has an Intel Core i7 CPU clocked at 2.8GHz with 8Gb of RAM. Each implementation has a maximum of 7,200 seconds (2 hours) to solve each individual instance.

Both implementations are using the  SCIP 4.0's API (see \cite{Maher2017}) and both are calling the LP solver of IBM ILOG Cplex 12.6.1.

\subsection{Instances\label{section:4.1}}

Since no standard libraries of instances for DOMP are available in public repositories we generate our own instances with the \textit{pseudorandom} number generator from the C random library.

We consider 9 sets of 30 instances. Each set has a different number of clients such that $n \in \{20,30,40,50,60,70,80,90,100\}$. For a given $n$, we generate one subset of 10 instances for each value of $p$, where $p \in \{ \left \lfloor{(n/4)} \right \rfloor, \left \lfloor{(n/3)} \right \rfloor, \left \lfloor{(n/2)} \right \rfloor \}$.

For a given $n$, we first randomly generate the Cartesian coordinates of the potential servers in the square $[0,400]^2$. Then, we calculate the cost for each pair of clients with the Euclidean distance between the two related nodes in the square. We round each distances to the nearest integer to build the cost matrices. We also fix the values of the matrix diagonal to the smallest admissible cost to avoid free self service.

Finally, we randomly generate the weighted ordered vector $\lambda$ such that, for each potential server $i=1,\dots, n$, $\lambda_i \in [n/4,n]$. The parameters for the generation process are given in table \ref{ParamGenInst}.

\begin{table}[!ht]
	\begin{center}
		\begin{tabular}{l||ccc|ccc|ccc|ccc|ccc}
			$n$			& \multicolumn{3}{c|}{20} & \multicolumn{3}{c|}{30} & \multicolumn{3}{c|}{40} & \multicolumn{3}{c|}{50} & \multicolumn{3}{c}{60}  \\ \hline
			$p$			& 5 & 6 & 10 & 7 & 10 & 15 & 10 & 13 & 20 & 12 & 16 & 25 & 15 & 20 & 30  \\ \hline
			$\lambda$	& \multicolumn{3}{c|}{$[5,20]^n$} & \multicolumn{3}{c|}{$[7,30]^n$} & \multicolumn{3}{c|}{$[10,40]^n$} & \multicolumn{3}{c|}{$[12,50]^n$} & \multicolumn{3}{c}{$[15,60]^n$}  \\ \hline\hline
			$n$			& \multicolumn{3}{c|}{70} & \multicolumn{3}{c|}{80} & \multicolumn{3}{c|}{90} & \multicolumn{3}{c|}{100}   \\ \hline
			$p$			& 17 & 23 & 35 & 20 & 26 & 40 & 22 & 30 & 45 & 25 & 33 & 50  \\ \hline
			$\lambda$	& \multicolumn{3}{c|}{$[17,60]^n$} & \multicolumn{3}{c|}{$[20,80]^n$} & \multicolumn{3}{c|}{$[22,90]^n$} & \multicolumn{3}{c|}{$[25,100]^n$}  \\

		\end{tabular}
	\caption{Parameters for the generation of the instances.} \label{ParamGenInst}
	\end{center}
\end{table}

All these instances, with $n$ up to 100, are available at \url{http://gom.ulb.ac.be/domp_repo/}.

\subsection{$MP$ vs $WOC$ linear relaxations \label{section:4.2}}

We assess experimentally the linear relaxation of $MP$ by comparing with $WOC$ on all the instances generated.  For these experiments, no cuts have been applied.

In Table \ref{ResultsLP20and30}, we report averages of the numerical results of the linear relaxation for both formulations. We report the values $GapLP(\%)$ which are the percentage gaps between the optimal integer values $z^*$ (alternatively the best known solution) and the linear relaxation optimal values $z^*_{LP}$ such that $GapLP(\%) = 100 (z^* - z^*_{LP})/ z^*$. We also report the computational times (in seconds).

Table \ref{ResultsLP20and30} also includes average number of variables ($|Vars|$) and required memory ($Memory(MB)$). The reader can see that, in terms of time, $MP$ has some room for improvement as compared with the professional implementation of Cplex used for solving $WOC$. On the contrary, we highlight the small number of variables that are used to certify optimality with this column generation approach $MP$.

\begin{table}[!ht]
	\begin{center}
	  \scriptsize
		\begin{tabular}{ll|rrr|rrr}
			&	& \multicolumn{3}{c|}{$n=20$} & \multicolumn{3}{c}{$n=30$}  \\
			&	& $p=5$ & $p=6$ & $p=10$ & $p=7$ & $p=10$ & $p=15$ \\ \hline	
			\textbf{B\&C} & $GapLP(\%)$    & 8.64	& 8.66 & 13.13	&	9.45 &	 10.38 & 14.28 \\
						
			(WOC) & $Time(s)$     & 0.14	&0.14& 0.12	&	0.70	&0.68 & 0.62 \\
			
			& $|Vars|$     & 8020	&8020&8020	&	27030	&27030 & 27030 \\
			
			& $Memory(MB)$     & 35	&35&35	&	101	&101 & 101 \\\hline

			\textbf{\BPC} & $GapLP(\%)$    & 	7.87 &	8.03 & 12.70	& 8.46 &	 9.81 & 13.83 \\	
				
			(MP) & $Time(s)$     & 1.19 &	0.89 & 0.63	&	6.04 &	3.98 & 3.74 \\
			
			& $|Vars|$     & 724	&656&537	&	1754	&1570 & 1484 \\
			
			& $Memory(MB)$     & 7	&6&4	&	20	&17 & 14 \\\hline
			
			&	& \multicolumn{3}{c|}{$n=40$} & \multicolumn{3}{c}{$n=50$}  \\
			&	& $p=10$ & $p=13$ & $p=20$ & $p=12$ & $p=16$ & $p=25$ \\ \hline	
			\textbf{B\&C} & $GapLP(\%)$    & 9.43	& 11.00 & 15.35	&	7.34 &	 8.77 & 12.97 \\
						
			(WOC) & $Time(s)$     & 2.51	&2.34& 2.09	&	7.35	&6.37 & 6.25 \\
			
			& $|Vars|$     & 64040	&64040&64040	&	125050	&125050 & 125050 \\
			
			& $Memory(MB)$     & 235	&235&235	&	451	&451 & 451 \\\hline

			\textbf{\BPC} & $GapLP(\%)$    & 	9.11 &	10.75 & 15.18& 6.98 &	 8.51 & 12.76 \\	
				
			(MP) & $Time(s)$     & 17.61 &	13.85 & 12.17	&	40.75 &	33.72 & 33.13 \\
			
			& $|Vars|$     & 3370	&3149&3111	&	5355	&5182 & 5175 \\
			
			& $Memory(MB)$     & 46	&39&35	&	82	&72 & 68 \\\hline
			
	&	& \multicolumn{3}{c|}{$n=60$} & \multicolumn{3}{c}{$n=70$}  \\
			&	& $p=15$ & $p=20$ & $p=30$ & $p=17$ & $p=23$ & $p=35$ \\ \hline	
			\textbf{B\&C} & $GapLP(\%)$    & 8.84	& 9.95 & 14.43	&	8.04 &	 9.19 & 13.73 \\
						
			(WOC) & $Time(s)$     & 15.98	&13.30& 12.27	&	40.78	 &35.65 & 29.40 \\
			
			& $|Vars|$     & 216060	&216060&216060	&	125050	&343070 & 343070 \\
			
			& $Memory(MB)$     & 764	&764&764	&	1214	&1214 & 1214 \\\hline

			\textbf{\BPC} & $GapLP(\%)$    & 	8.56 &	9.71 & 14.25	& 7.79 &	 9.04 & 13.62 \\	
				
			(MP) & $Time(s)$     & 94.79 &	72.47 & 92.70	&	176.21 &	 157.19 & 212.97 \\
			
			& $|Vars|$     & 8146	&7592&9069	&	11112	&11250 & 13648 \\
			
			& $Memory(MB)$     & 139	&120&142	&	211	&202 & 244 \\\hline
			
				&	& \multicolumn{3}{c|}{$n=80$} & \multicolumn{3}{c}{$n=90$}  \\
			&	& $p=20$ & $p=26$ & $p=40$ & $p=22$ & $p=30$ & $p=45$ \\ \hline	
			\textbf{B\&C} & $GapLP(\%)$    & 8.65	& 7.65 & 7.12	&	8.70 &	 6.60 & 6.69 \\
						
			(WOC) & $Time(s)$     & 67.42	&58.14& 47.63	&	128.70	 &96.74 & 82.19 \\
			
			& $|Vars|$     & 512080	&512080&512080	&	729090	&729090 & 729090 \\
			
			& $Memory(MB)$     & 1830	&1830&1830	&	2561	&2561 & 2561 \\\hline

			\textbf{\BPC} & $GapLP(\%)$    & 	8.53 &	7.48 & 7.08	& 8.55 &	 6.55 & 6.66 \\	
				
			(MP) & $Time(s)$     & 352.75 &	264.60 & 210.92	&	713.28 &	 459.07 & 404.18 \\
			
			& $|Vars|$     & 15704	&14163&11851	&	21566	&18451 & 16205 \\
			
			& $Memory(MB)$     & 330	&280&214	&	513	&404 & 336 \\\hline
			
				&	& \multicolumn{3}{c|}{$n=100$}  \\
			&	& $p=25$ & $p=33$ & $p=50$  \\ \hline	
			\textbf{B\&C} & $GapLP(\%)$    & --	& -- & --	 \\
						
			(WOC) & $Time(s)$     & --	&--& --	 \\
			
			& $|Vars|$     & 1000100	&1000100&1000100	 \\
			
			& $Memory(MB)$     & $>$4096	&$>$4096&$>$4096	\\\hline

			\textbf{\BPC} & $GapLP(\%)$    & 	7.94 &	7.40 & 6.59	 \\	
				
			(MP) & $Time(s)$     & 1417.65 &	939.40 & 667.86	 \\
			
			& $|Vars|$     & 30202	&26068&21101	 \\
			
			& $Memory(MB)$     & 809	&656&482 \\\hline

		\end{tabular}
		\caption{Numerical results on linear relaxation for $WOC$ and $MP$ } \label{ResultsLP20and30}
	\end{center}
\end{table}

As expected, according to Proposition \ref{c3-pro-inclusionMP-DOMP}, the integrality gap of formulation $MP$ outperforms the one by $WOC$. Moreover, formulation $MP$ also outperforms  $WOC$ in number of required variables (see Figure \ref{performanceProfiles1}) which results in much smaller memory requirements (see Figure \ref{performanceProfiles2}). Indeed, the implementation of $WOC$ fails to solve, already for sizes of $n=100$, the linear relaxation of all instances by lack of RAM memory; whereas with the same parameter configuration, formulation $MP$ does not experience that problem. Figure \ref{performanceProfiles2} shows the performance profile of the memory requirement of both formulations. As one can see  \textbf{B\&P\&C}(MP) outperforms $WOC$ with respect to this factor for all instance sizes.

\begin{figure}
\centering
\begin{minipage}[!ht]{0.45\textwidth}
	\begin{tikzpicture}[scale=0.7,font=\footnotesize]
	\begin{axis}[axis x line=bottom,  axis y line=left,
	xlabel=$n$,
	ylabel=Number of variables,
	legend style={legend pos=south east,font=\scriptsize}]

	\addplot[orange,semithick,dash pattern=on 4pt off 2pt] plot coordinates {		
	
(	20	,	8020	)
(	20	,	8020	)
(	20	,	8020	)
(	20	,	8020	)
(	20	,	8020	)
(	20	,	8020	)
(	20	,	8020	)
(	20	,	8020	)
(	20	,	8020	)
(	20	,	8020	)
(	20	,	8020	)
(	20	,	8020	)
(	20	,	8020	)
(	20	,	8020	)
(	20	,	8020	)
(	20	,	8020	)
(	20	,	8020	)
(	20	,	8020	)
(	20	,	8020	)
(	20	,	8020	)
(	20	,	8020	)
(	20	,	8020	)
(	20	,	8020	)
(	20	,	8020	)
(	20	,	8020	)
(	20	,	8020	)
(	20	,	8020	)
(	20	,	8020	)
(	20	,	8020	)
(	20	,	8020	)
(	30	,	27030	)
(	30	,	27030	)
(	30	,	27030	)
(	30	,	27030	)
(	30	,	27030	)
(	30	,	27030	)
(	30	,	27030	)
(	30	,	27030	)
(	30	,	27030	)
(	30	,	27030	)
(	30	,	27030	)
(	30	,	27030	)
(	30	,	27030	)
(	30	,	27030	)
(	30	,	27030	)
(	30	,	27030	)
(	30	,	27030	)
(	30	,	27030	)
(	30	,	27030	)
(	30	,	27030	)
(	30	,	27030	)
(	30	,	27030	)
(	30	,	27030	)
(	30	,	27030	)
(	30	,	27030	)
(	30	,	27030	)
(	30	,	27030	)
(	30	,	27030	)
(	30	,	27030	)
(	30	,	27030	)
(	40	,	64040	)
(	40	,	64040	)
(	40	,	64040	)
(	40	,	64040	)
(	40	,	64040	)
(	40	,	64040	)
(	40	,	64040	)
(	40	,	64040	)
(	40	,	64040	)
(	40	,	64040	)
(	40	,	64040	)
(	40	,	64040	)
(	40	,	64040	)
(	40	,	64040	)
(	40	,	64040	)
(	40	,	64040	)
(	40	,	64040	)
(	40	,	64040	)
(	40	,	64040	)
(	40	,	64040	)
(	40	,	64040	)
(	40	,	64040	)
(	40	,	64040	)
(	40	,	64040	)
(	40	,	64040	)
(	40	,	64040	)
(	40	,	64040	)
(	40	,	64040	)
(	40	,	64040	)
(	40	,	64040	)
(	50	,	125050	)
(	50	,	125050	)
(	50	,	125050	)
(	50	,	125050	)
(	50	,	125050	)
(	50	,	125050	)
(	50	,	125050	)
(	50	,	125050	)
(	50	,	125050	)
(	50	,	125050	)
(	50	,	125050	)
(	50	,	125050	)
(	50	,	125050	)
(	50	,	125050	)
(	50	,	125050	)
(	50	,	125050	)
(	50	,	125050	)
(	50	,	125050	)
(	50	,	125050	)
(	50	,	125050	)
(	50	,	125050	)
(	50	,	125050	)
(	50	,	125050	)
(	50	,	125050	)
(	50	,	125050	)
(	50	,	125050	)
(	50	,	125050	)
(	50	,	125050	)
(	50	,	125050	)
(	50	,	125050	)
(	60	,	216060	)
(	60	,	216060	)
(	60	,	216060	)
(	60	,	216060	)
(	60	,	216060	)
(	60	,	216060	)
(	60	,	216060	)
(	60	,	216060	)
(	60	,	216060	)
(	60	,	216060	)
(	60	,	216060	)
(	60	,	216060	)
(	60	,	216060	)
(	60	,	216060	)
(	60	,	216060	)
(	60	,	216060	)
(	60	,	216060	)
(	60	,	216060	)
(	60	,	216060	)
(	60	,	216060	)
(	60	,	216060	)
(	60	,	216060	)
(	60	,	216060	)
(	60	,	216060	)
(	60	,	216060	)
(	60	,	216060	)
(	60	,	216060	)
(	60	,	216060	)
(	60	,	216060	)
(	60	,	216060	)
(	70	,	343070	)
(	70	,	343070	)
(	70	,	343070	)
(	70	,	343070	)
(	70	,	343070	)
(	70	,	343070	)
(	70	,	343070	)
(	70	,	343070	)
(	70	,	343070	)
(	70	,	343070	)
(	70	,	343070	)
(	70	,	343070	)
(	70	,	343070	)
(	70	,	343070	)
(	70	,	343070	)
(	70	,	343070	)
(	70	,	343070	)
(	70	,	343070	)
(	70	,	343070	)
(	70	,	343070	)
(	70	,	343070	)
(	70	,	343070	)
(	70	,	343070	)
(	70	,	343070	)
(	70	,	343070	)
(	70	,	343070	)
(	70	,	343070	)
(	70	,	343070	)
(	70	,	343070	)
(	70	,	343070	)
(	80	,	512080	)
(	80	,	512080	)
(	80	,	512080	)
(	80	,	512080	)
(	80	,	512080	)
(	80	,	512080	)
(	80	,	512080	)
(	80	,	512080	)
(	80	,	512080	)
(	80	,	512080	)
(	80	,	512080	)
(	80	,	512080	)
(	80	,	512080	)
(	80	,	512080	)
(	80	,	512080	)
(	80	,	512080	)
(	80	,	512080	)
(	80	,	512080	)
(	80	,	512080	)
(	80	,	512080	)
(	80	,	512080	)
(	80	,	512080	)
(	80	,	512080	)
(	80	,	512080	)
(	80	,	512080	)
(	80	,	512080	)
(	80	,	512080	)
(	80	,	512080	)
(	80	,	512080	)
(	80	,	512080	)
(	90	,	729090	)
(	90	,	729090	)
(	90	,	729090	)
(	90	,	729090	)
(	90	,	729090	)
(	90	,	729090	)
(	90	,	729090	)
(	90	,	729090	)
(	90	,	729090	)
(	90	,	729090	)
(	90	,	729090	)
(	90	,	729090	)
(	90	,	729090	)
(	90	,	729090	)
(	90	,	729090	)
(	90	,	729090	)
(	90	,	729090	)
(	90	,	729090	)
(	90	,	729090	)
(	90	,	729090	)
(	90	,	729090	)
(	90	,	729090	)
(	90	,	729090	)
(	90	,	729090	)
(	90	,	729090	)
(	90	,	729090	)
(	90	,	729090	)
(	90	,	729090	)
(	90	,	729090	)
(	90	,	729090	)
(	100	,	1000100	)
(	100	,	1000100	)
(	100	,	1000100	)
(	100	,	1000100	)
(	100	,	1000100	)
(	100	,	1000100	)
(	100	,	1000100	)
(	100	,	1000100	)
(	100	,	1000100	)
(	100	,	1000100	)
(	100	,	1000100	)
(	100	,	1000100	)
(	100	,	1000100	)
(	100	,	1000100	)
(	100	,	1000100	)
(	100	,	1000100	)
(	100	,	1000100	)
(	100	,	1000100	)
(	100	,	1000100	)
(	100	,	1000100	)
(	100	,	1000100	)
(	100	,	1000100	)
(	100	,	1000100	)
(	100	,	1000100	)
(	100	,	1000100	)
(	100	,	1000100	)
(	100	,	1000100	)
(	100	,	1000100	)
(	100	,	1000100	)
(	100	,	1000100	)
		
	};
	\addlegendentry{\textbf{B\&P\&C}(MP)}

	\addplot[blue,semithick,dash pattern=on 4pt off 2pt] plot coordinates {		
	
	(	20	,	374	)
	(	20	,	412	)
	(	20	,	423	)
	(	20	,	426	)
	(	20	,	426	)
	(	20	,	433	)
	(	20	,	438	)
	(	20	,	440	)
	(	20	,	465	)
	(	20	,	482	)
	(	20	,	542	)
	(	20	,	572	)
	(	20	,	573	)
	(	20	,	591	)
	(	20	,	595	)
	(	20	,	596	)
	(	20	,	615	)
	(	20	,	624	)
	(	20	,	628	)
	(	20	,	679	)
	(	20	,	683	)
	(	20	,	723	)
	(	20	,	733	)
	(	20	,	739	)
	(	20	,	750	)
	(	20	,	758	)
	(	20	,	769	)
	(	20	,	775	)
	(	20	,	814	)
	(	20	,	824	)
	(	20	,	929	)
	(	20.33	,	1014	)
	(	20.66	,	1095	)
	(	20.99	,	1154	)
	(	21.32	,	1159	)
	(	21.65	,	1185	)
	(	21.98	,	1263	)
	(	22.31	,	1280	)
	(	22.64	,	1285	)
	(	22.97	,	1304	)
	(	23.3	,	1328	)
	(	23.63	,	1335	)
	(	23.96	,	1360	)
	(	24.29	,	1418	)
	(	24.62	,	1424	)
	(	24.95	,	1430	)
	(	25.28	,	1444	)
	(	25.61	,	1505	)
	(	25.94	,	1522	)
	(	26.27	,	1529	)
	(	26.6	,	1672	)
	(	26.93	,	1705	)
	(	27.26	,	1748	)
	(	27.59	,	1818	)
	(	27.92	,	1874	)
	(	28.25	,	1875	)
	(	28.58	,	1882	)
	(	28.91	,	1883	)
	(	29.24	,	1969	)
	(	29.57	,	1984	)
	(	30	,	2066	)
	(	30.33	,	2139	)
	(	30.66	,	2159	)
	(	30.99	,	2204	)
	(	31.32	,	2359	)
	(	31.65	,	2428	)
	(	31.98	,	2467	)
	(	32.31	,	2486	)
	(	32.64	,	2507	)
	(	32.97	,	2611	)
	(	33.3	,	2716	)
	(	33.63	,	2811	)
	(	33.96	,	2811	)
	(	34.29	,	2856	)
	(	34.62	,	2860	)
	(	34.95	,	2873	)
	(	35.28	,	2877	)
	(	35.61	,	2896	)
	(	35.94	,	2898	)
	(	36.27	,	2945	)
	(	36.6	,	2951	)
	(	36.93	,	2956	)
	(	37.26	,	3002	)
	(	37.59	,	3013	)
	(	37.92	,	3026	)
	(	38.25	,	3055	)
	(	38.58	,	3369	)
	(	38.91	,	3378	)
	(	39.24	,	3398	)
	(	39.57	,	3438	)
	(	40	,	3446	)
	(	40.33	,	3529	)
	(	40.66	,	3721	)
	(	40.99	,	3741	)
	(	41.32	,	3827	)
	(	41.65	,	3964	)
	(	41.98	,	3990	)
	(	42.31	,	4168	)
	(	42.64	,	4337	)
	(	42.97	,	4363	)
	(	43.3	,	4368	)
	(	43.63	,	4411	)
	(	43.96	,	4458	)
	(	44.29	,	4481	)
	(	44.62	,	4482	)
	(	44.95	,	4529	)
	(	45.28	,	4577	)
	(	45.61	,	4586	)
	(	45.94	,	4901	)
	(	46.27	,	4906	)
	(	46.6	,	5098	)
	(	46.93	,	5168	)
	(	47.26	,	5221	)
	(	47.59	,	5345	)
	(	47.92	,	5409	)
	(	48.25	,	5415	)
	(	48.58	,	5546	)
	(	48.91	,	5572	)
	(	49.24	,	5604	)
	(	49.57	,	5732	)
	(	50	,	5921	)
	(	50.33	,	5926	)
	(	50.66	,	5953	)
	(	50.99	,	6006	)
	(	51.32	,	6076	)
	(	51.65	,	6093	)
	(	51.98	,	6136	)
	(	52.31	,	6176	)
	(	52.64	,	6219	)
	(	52.97	,	6325	)
	(	53.3	,	6372	)
	(	53.63	,	6625	)
	(	53.96	,	6693	)
	(	54.29	,	6793	)
	(	54.62	,	6909	)
	(	54.95	,	6962	)
	(	55.28	,	7103	)
	(	55.61	,	7161	)
	(	55.94	,	7200	)
	(	56.27	,	7260	)
	(	56.6	,	7386	)
	(	56.93	,	7386	)
	(	57.26	,	7387	)
	(	57.59	,	7405	)
	(	57.92	,	7474	)
	(	58.25	,	7533	)
	(	58.58	,	7610	)
	(	58.91	,	7662	)
	(	59.24	,	7722	)
	(	59.57	,	7765	)
	(	60	,	7844	)
	(	60.33	,	7874	)
	(	60.66	,	8034	)
	(	60.99	,	8083	)
	(	61.32	,	8208	)
	(	61.65	,	8256	)
	(	61.98	,	8272	)
	(	62.31	,	8334	)
	(	62.64	,	8449	)
	(	62.97	,	8510	)
	(	63.3	,	8584	)
	(	63.63	,	8589	)
	(	63.96	,	8779	)
	(	64.29	,	8797	)
	(	64.62	,	8868	)
	(	64.95	,	8891	)
	(	65.28	,	8920	)
	(	65.61	,	8960	)
	(	65.94	,	9070	)
	(	66.27	,	9233	)
	(	66.6	,	9260	)
	(	66.93	,	9310	)
	(	67.26	,	9320	)
	(	67.59	,	9442	)
	(	67.92	,	9473	)
	(	68.25	,	9552	)
	(	68.58	,	9716	)
	(	68.91	,	9771	)
	(	69.24	,	9775	)
	(	69.57	,	10121	)
	(	70	,	10142	)
	(	70.33	,	10200	)
	(	70.66	,	10282	)
	(	70.99	,	10307	)
	(	71.32	,	10311	)
	(	71.65	,	10403	)
	(	71.98	,	10492	)
	(	72.31	,	10593	)
	(	72.64	,	10656	)
	(	72.97	,	10683	)
	(	73.3	,	10735	)
	(	73.63	,	11067	)
	(	73.96	,	11075	)
	(	74.29	,	11174	)
	(	74.62	,	11268	)
	(	74.95	,	11356	)
	(	75.28	,	11364	)
	(	75.61	,	11481	)
	(	75.94	,	11563	)
	(	76.27	,	11680	)
	(	76.6	,	11784	)
	(	76.93	,	11946	)
	(	77.26	,	11967	)
	(	77.59	,	12160	)
	(	77.92	,	12167	)
	(	78.25	,	12194	)
	(	78.58	,	12212	)
	(	78.91	,	12240	)
	(	79.24	,	12272	)
	(	79.57	,	12387	)
	(	80	,	12642	)
	(	80.33	,	12836	)
	(	80.66	,	13052	)
	(	80.99	,	13274	)
	(	81.32	,	13468	)
	(	81.65	,	13552	)
	(	81.98	,	13765	)
	(	82.31	,	13863	)
	(	82.64	,	13916	)
	(	82.97	,	14072	)
	(	83.3	,	14117	)
	(	83.63	,	14323	)
	(	83.96	,	14326	)
	(	84.29	,	14335	)
	(	84.62	,	14523	)
	(	84.95	,	14560	)
	(	85.28	,	15511	)
	(	85.61	,	15641	)
	(	85.94	,	15764	)
	(	86.27	,	16069	)
	(	86.6	,	16275	)
	(	86.93	,	16297	)
	(	87.26	,	16336	)
	(	87.59	,	16444	)
	(	87.92	,	16489	)
	(	88.25	,	16904	)
	(	88.58	,	17209	)
	(	88.91	,	17442	)
	(	89.24	,	17657	)
	(	89.57	,	17854	)
	(	90	,	18269	)
	(	90.33	,	18426	)
	(	90.66	,	18443	)
	(	90.99	,	19048	)
	(	91.32	,	19298	)
	(	91.65	,	19342	)
	(	91.98	,	19657	)
	(	92.31	,	19846	)
	(	92.64	,	20148	)
	(	92.97	,	20655	)
	(	93.3	,	21559	)
	(	93.63	,	22745	)
	(	93.96	,	22787	)
	(	94.29	,	23070	)
	(	94.62	,	23558	)
	(	94.95	,	24417	)
	(	95.28	,	24466	)
	(	95.61	,	24545	)
	(	95.94	,	24691	)
	(	96.27	,	25949	)
	(	96.6	,	26557	)
	(	96.93	,	26872	)
	(	97.26	,	27573	)
	(	97.59	,	28157	)
	(	97.92	,	30632	)
	(	98.25	,	30657	)
	(	98.58	,	33040	)
	(	98.91	,	33041	)
	(	99.24	,	34001	)
	(	99.57	,	35380	)
		};
		\addlegendentry{\textbf{B\&C}(WOC)}
	
	\end{axis}
	\end{tikzpicture}
	\caption{\small Graph of Number of Variables versus size $n$ for \textbf{B\&C}(WOC) and \textbf{B\&P\&C}(MP).} \label{performanceProfiles1}
\end{minipage}
\begin{minipage}[c]{0.45\textwidth}
	\begin{tikzpicture}[scale=0.7,font=\footnotesize]
	\begin{axis}[axis x line=bottom,  axis y line=left,
	xlabel=$n$,
	ylabel=Memory (Mb),
	legend style={legend pos=north west,font=\scriptsize}]

	\addplot[orange,semithick,dash pattern=on 4pt off 2pt] plot coordinates {		

(	20	,	35	)
(	20	,	35	)
(	20	,	35	)
(	20	,	35	)
(	20	,	35	)
(	20	,	35	)
(	20	,	35	)
(	20	,	35	)
(	20	,	35	)
(	20	,	35	)
(	20	,	35	)
(	20	,	35	)
(	20	,	35	)
(	20	,	35	)
(	20	,	35	)
(	20	,	35	)
(	20	,	35	)
(	20	,	35	)
(	20	,	35	)
(	20	,	35	)
(	20	,	35	)
(	20	,	35	)
(	20	,	35	)
(	20	,	35	)
(	20	,	35	)
(	20	,	35	)
(	20	,	35	)
(	20	,	35	)
(	20	,	35	)
(	20	,	35	)
(	30	,	101	)
(	30	,	101	)
(	30	,	101	)
(	30	,	101	)
(	30	,	101	)
(	30	,	101	)
(	30	,	101	)
(	30	,	101	)
(	30	,	101	)
(	30	,	101	)
(	30	,	101	)
(	30	,	101	)
(	30	,	101	)
(	30	,	101	)
(	30	,	101	)
(	30	,	101	)
(	30	,	101	)
(	30	,	101	)
(	30	,	101	)
(	30	,	101	)
(	30	,	101	)
(	30	,	101	)
(	30	,	101	)
(	30	,	101	)
(	30	,	101	)
(	30	,	101	)
(	30	,	101	)
(	30	,	101	)
(	30	,	101	)
(	30	,	101	)
(	40	,	235	)
(	40	,	235	)
(	40	,	235	)
(	40	,	235	)
(	40	,	235	)
(	40	,	235	)
(	40	,	235	)
(	40	,	235	)
(	40	,	235	)
(	40	,	235	)
(	40	,	235	)
(	40	,	235	)
(	40	,	235	)
(	40	,	235	)
(	40	,	235	)
(	40	,	235	)
(	40	,	235	)
(	40	,	235	)
(	40	,	235	)
(	40	,	235	)
(	40	,	235	)
(	40	,	235	)
(	40	,	235	)
(	40	,	235	)
(	40	,	235	)
(	40	,	235	)
(	40	,	235	)
(	40	,	235	)
(	40	,	235	)
(	40	,	235	)
(	50	,	451	)
(	50	,	451	)
(	50	,	451	)
(	50	,	451	)
(	50	,	451	)
(	50	,	451	)
(	50	,	451	)
(	50	,	451	)
(	50	,	451	)
(	50	,	451	)
(	50	,	451	)
(	50	,	451	)
(	50	,	451	)
(	50	,	451	)
(	50	,	451	)
(	50	,	451	)
(	50	,	451	)
(	50	,	451	)
(	50	,	451	)
(	50	,	451	)
(	50	,	451	)
(	50	,	451	)
(	50	,	451	)
(	50	,	451	)
(	50	,	451	)
(	50	,	451	)
(	50	,	451	)
(	50	,	451	)
(	50	,	451	)
(	50	,	451	)
(	60	,	764	)
(	60	,	764	)
(	60	,	764	)
(	60	,	764	)
(	60	,	764	)
(	60	,	764	)
(	60	,	764	)
(	60	,	764	)
(	60	,	764	)
(	60	,	764	)
(	60	,	764	)
(	60	,	764	)
(	60	,	764	)
(	60	,	764	)
(	60	,	764	)
(	60	,	764	)
(	60	,	764	)
(	60	,	764	)
(	60	,	764	)
(	60	,	764	)
(	60	,	764	)
(	60	,	764	)
(	60	,	764	)
(	60	,	764	)
(	60	,	764	)
(	60	,	764	)
(	60	,	764	)
(	60	,	764	)
(	60	,	764	)
(	60	,	764	)
(	70	,	1214	)
(	70	,	1214	)
(	70	,	1214	)
(	70	,	1214	)
(	70	,	1214	)
(	70	,	1214	)
(	70	,	1214	)
(	70	,	1214	)
(	70	,	1214	)
(	70	,	1214	)
(	70	,	1214	)
(	70	,	1214	)
(	70	,	1214	)
(	70	,	1214	)
(	70	,	1214	)
(	70	,	1214	)
(	70	,	1214	)
(	70	,	1214	)
(	70	,	1214	)
(	70	,	1214	)
(	70	,	1214	)
(	70	,	1214	)
(	70	,	1214	)
(	70	,	1214	)
(	70	,	1214	)
(	70	,	1214	)
(	70	,	1214	)
(	70	,	1214	)
(	70	,	1214	)
(	70	,	1214	)
(	80	,	1830	)
(	80	,	1830	)
(	80	,	1830	)
(	80	,	1830	)
(	80	,	1830	)
(	80	,	1830	)
(	80	,	1830	)
(	80	,	1830	)
(	80	,	1830	)
(	80	,	1830	)
(	80	,	1830	)
(	80	,	1830	)
(	80	,	1830	)
(	80	,	1830	)
(	80	,	1830	)
(	80	,	1830	)
(	80	,	1830	)
(	80	,	1830	)
(	80	,	1830	)
(	80	,	1830	)
(	80	,	1830	)
(	80	,	1830	)
(	80	,	1830	)
(	80	,	1830	)
(	80	,	1830	)
(	80	,	1830	)
(	80	,	1830	)
(	80	,	1830	)
(	80	,	1830	)
(	80	,	1830	)
(	90	,	2561	)
(	90	,	2561	)
(	90	,	2561	)
(	90	,	2561	)
(	90	,	2561	)
(	90	,	2561	)
(	90	,	2561	)
(	90	,	2561	)
(	90	,	2561	)
(	90	,	2561	)
(	90	,	2561	)
(	90	,	2561	)
(	90	,	2561	)
(	90	,	2561	)
(	90	,	2561	)
(	90	,	2561	)
(	90	,	2561	)
(	90	,	2561	)
(	90	,	2561	)
(	90	,	2561	)
(	90	,	2561	)
(	90	,	2561	)
(	90	,	2561	)
(	90	,	2561	)
(	90	,	2561	)
(	90	,	2561	)
(	90	,	2561	)
(	90	,	2561	)
(	90	,	2561	)
(	90	,	2561	)
	
	};
	\addlegendentry{\textbf{B\&C}(WOC)}

	\addplot[blue,semithick,dash pattern=on 4pt off 2pt] plot coordinates {		
	
(	20	,	3	)
(	20	,	3	)
(	20	,	4	)
(	20	,	4	)
(	20	,	4	)
(	20	,	4	)
(	20	,	4	)
(	20	,	4	)
(	20	,	4	)
(	20	,	4	)
(	20	,	4	)
(	20	,	4	)
(	20	,	4	)
(	20	,	5	)
(	20	,	5	)
(	20	,	5	)
(	20	,	6	)
(	20	,	6	)
(	20	,	6	)
(	20	,	6	)
(	20	,	6	)
(	20	,	7	)
(	20	,	7	)
(	20	,	7	)
(	20	,	7	)
(	20	,	7	)
(	20	,	7	)
(	20	,	7	)
(	20	,	8	)
(	20	,	8	)
(	20	,	8	)
(	20.33	,	9	)
(	20.66	,	10	)
(	20.99	,	11	)
(	21.32	,	12	)
(	21.65	,	12	)
(	21.98	,	12	)
(	22.31	,	12	)
(	22.64	,	12	)
(	22.97	,	12	)
(	23.3	,	13	)
(	23.63	,	13	)
(	23.96	,	14	)
(	24.29	,	15	)
(	24.62	,	15	)
(	24.95	,	15	)
(	25.28	,	15	)
(	25.61	,	15	)
(	25.94	,	15	)
(	26.27	,	16	)
(	26.6	,	18	)
(	26.93	,	19	)
(	27.26	,	20	)
(	27.59	,	20	)
(	27.92	,	21	)
(	28.25	,	21	)
(	28.58	,	21	)
(	28.91	,	22	)
(	29.24	,	22	)
(	29.57	,	22	)
(	30	,	22	)
(	30.33	,	23	)
(	30.66	,	23	)
(	30.99	,	24	)
(	31.32	,	25	)
(	31.65	,	26	)
(	31.98	,	27	)
(	32.31	,	28	)
(	32.64	,	28	)
(	32.97	,	31	)
(	33.3	,	32	)
(	33.63	,	32	)
(	33.96	,	33	)
(	34.29	,	33	)
(	34.62	,	33	)
(	34.95	,	33	)
(	35.28	,	34	)
(	35.61	,	34	)
(	35.94	,	35	)
(	36.27	,	35	)
(	36.6	,	35	)
(	36.93	,	35	)
(	37.26	,	36	)
(	37.59	,	36	)
(	37.92	,	39	)
(	38.25	,	41	)
(	38.58	,	41	)
(	38.91	,	42	)
(	39.24	,	43	)
(	39.57	,	44	)
(	40	,	44	)
(	40.33	,	47	)
(	40.66	,	49	)
(	40.99	,	49	)
(	41.32	,	49	)
(	41.65	,	49	)
(	41.98	,	50	)
(	42.31	,	51	)
(	42.64	,	54	)
(	42.97	,	58	)
(	43.3	,	59	)
(	43.63	,	59	)
(	43.96	,	61	)
(	44.29	,	62	)
(	44.62	,	62	)
(	44.95	,	63	)
(	45.28	,	65	)
(	45.61	,	65	)
(	45.94	,	69	)
(	46.27	,	71	)
(	46.6	,	72	)
(	46.93	,	73	)
(	47.26	,	76	)
(	47.59	,	79	)
(	47.92	,	80	)
(	48.25	,	80	)
(	48.58	,	81	)
(	48.91	,	81	)
(	49.24	,	84	)
(	49.57	,	86	)
(	50	,	87	)
(	50.33	,	87	)
(	50.66	,	88	)
(	50.99	,	88	)
(	51.32	,	89	)
(	51.65	,	91	)
(	51.98	,	91	)
(	52.31	,	91	)
(	52.64	,	93	)
(	52.97	,	95	)
(	53.3	,	96	)
(	53.63	,	97	)
(	53.96	,	98	)
(	54.29	,	104	)
(	54.62	,	106	)
(	54.95	,	106	)
(	55.28	,	107	)
(	55.61	,	111	)
(	55.94	,	115	)
(	56.27	,	115	)
(	56.6	,	116	)
(	56.93	,	119	)
(	57.26	,	119	)
(	57.59	,	119	)
(	57.92	,	119	)
(	58.25	,	122	)
(	58.58	,	126	)
(	58.91	,	126	)
(	59.24	,	126	)
(	59.57	,	127	)
(	60	,	130	)
(	60.33	,	131	)
(	60.66	,	132	)
(	60.99	,	134	)
(	61.32	,	135	)
(	61.65	,	136	)
(	61.98	,	138	)
(	62.31	,	139	)
(	62.64	,	145	)
(	62.97	,	146	)
(	63.3	,	146	)
(	63.63	,	147	)
(	63.96	,	151	)
(	64.29	,	155	)
(	64.62	,	155	)
(	64.95	,	155	)
(	65.28	,	156	)
(	65.61	,	157	)
(	65.94	,	157	)
(	66.27	,	159	)
(	66.6	,	167	)
(	66.93	,	168	)
(	67.26	,	168	)
(	67.59	,	169	)
(	67.92	,	171	)
(	68.25	,	173	)
(	68.58	,	173	)
(	68.91	,	174	)
(	69.24	,	178	)
(	69.57	,	178	)
(	70	,	178	)
(	70.33	,	179	)
(	70.66	,	180	)
(	70.99	,	185	)
(	71.32	,	185	)
(	71.65	,	190	)
(	71.98	,	192	)
(	72.31	,	195	)
(	72.64	,	198	)
(	72.97	,	199	)
(	73.3	,	201	)
(	73.63	,	204	)
(	73.96	,	207	)
(	74.29	,	211	)
(	74.62	,	212	)
(	74.95	,	212	)
(	75.28	,	213	)
(	75.61	,	216	)
(	75.94	,	221	)
(	76.27	,	224	)
(	76.6	,	225	)
(	76.93	,	225	)
(	77.26	,	229	)
(	77.59	,	230	)
(	77.92	,	232	)
(	78.25	,	235	)
(	78.58	,	238	)
(	78.91	,	241	)
(	79.24	,	252	)
(	79.57	,	254	)
(	80	,	255	)
(	80.33	,	260	)
(	80.66	,	261	)
(	80.99	,	267	)
(	81.32	,	270	)
(	81.65	,	271	)
(	81.98	,	280	)
(	82.31	,	283	)
(	82.64	,	285	)
(	82.97	,	288	)
(	83.3	,	288	)
(	83.63	,	297	)
(	83.96	,	299	)
(	84.29	,	302	)
(	84.62	,	306	)
(	84.95	,	308	)
(	85.28	,	329	)
(	85.61	,	332	)
(	85.94	,	334	)
(	86.27	,	347	)
(	86.6	,	356	)
(	86.93	,	365	)
(	87.26	,	367	)
(	87.59	,	367	)
(	87.92	,	370	)
(	88.25	,	372	)
(	88.58	,	376	)
(	88.91	,	377	)
(	89.24	,	382	)
(	89.57	,	387	)
(	90	,	409	)
(	90.33	,	423	)
(	90.66	,	424	)
(	90.99	,	429	)
(	91.32	,	430	)
(	91.65	,	438	)
(	91.98	,	440	)
(	92.31	,	446	)
(	92.64	,	465	)
(	92.97	,	468	)
(	93.3	,	514	)
(	93.63	,	553	)
(	93.96	,	559	)
(	94.29	,	578	)
(	94.62	,	581	)
(	94.95	,	611	)
(	95.28	,	614	)
(	95.61	,	618	)
(	95.94	,	626	)
(	96.27	,	626	)
(	96.6	,	655	)
(	96.93	,	685	)
(	97.26	,	688	)
(	97.59	,	696	)
(	97.92	,	798	)
(	98.25	,	821	)
(	98.58	,	888	)
(	98.91	,	928	)
(	99.24	,	959	)
(   100	,	1002	)

	};
	\addlegendentry{\textbf{B\&P\&C}(MP)}

	\end{axis}
	\end{tikzpicture}
	\caption{\small Graph of Memory usage (Mb) versus size  $n$ for \textbf{B\&C}(WOC) and \textbf{B\&P\&C}(MP).} \label{performanceProfiles2}
\end{minipage}
\end{figure}

\subsection{\BPC(MP) vs B\&C (WOC)\label{section:4.3}}

We now compare the \BPC\;  implementation of $MP$ with the B\&C implementation of $WOC$. The former is a branch-price-and-cut algorithm and the latter a branch-and-cut.

\begin{table}[!ht]
	\begin{center}
	  \scriptsize
		\begin{tabular}{ll|rrr|rrr}
			&	& \multicolumn{3}{c|}{$n=20$} & \multicolumn{3}{c}{$n=30$}  \\
			&	& $p=5$ & $p=6$ & $p=10$ & $p=7$ & $p=10$ & $p=15$ \\ \hline	
			\textbf{B\&C} &   $Time(s)$&16.54&11.50&4.48&1807.41&1578.21 &131.89  \\
						
			(WOC) & $|Vars|$&6054&5706&4211&20643&18245&13952\\
			
			& $|Nodes|$&1215&440&38&198424&305595&19197  \\
			
			& $|Cuts|$&1537&1249&689&4789&3056 &2519  \\

                   & $\# unsolved (T/M)$&0/0&0/0&0/0&1/1&1/0 &0/0  \\

                   & $Gap(\%)$&0.00&0.00&0.00&0.63&0.12 &0.00  \\\hline

			\textbf{\BPC} & $Time(s)$&3425.38&2220.55&159.35&6011.22&6298.75&4849.08\\	
				
			(MP) & $|Vars|$&13477&9054&4451&9493&11427 &11464  \\
			
			& $|Nodes|$&24&21&54&2&15 &26  \\
			
			& $|Cuts|$&1289&1028&543&3945&2520 &2162  \\

                   & $\# unsolved$&2&1&0&8&8 &6  \\

                   & $Gap(\%)$&0.45&0.14&0.00&1.38&1.18 &0.90  \\\hline

			&	& \multicolumn{3}{c|}{$n=40$} & \multicolumn{3}{c}{$n=50$}  \\
			&	& $p=10$ & $p=13$ & $p=20$ & $p=12$ & $p=16$ & $p=25$ \\ \hline				 \textbf{B\&C} &   $Time(s)$&7050.93&7061.36&6202.85&7200.00&7116.54 &6575.59  \\
						
			(WOC) & $|Vars|$&48065&43664&32820&94784&85630 &63776  \\
			
			& $|Nodes|$&602685&628962&605812&270959&284028 &355560  \\
			
			& $|Cuts|$&7939&6559&4727&12579&10423 &10131  \\

                   & $\# unsolved (T/M)$&7/3&8/2&8/0&10/0&9/1 &9/0  \\

                   & $Gap(\%)$&1.65&2.30&2.45&0.90&1.13 &1.32  \\\hline

			\textbf{\BPC} & $Time(s)$&7200.00&6572.81&6709.53&7200.00&7200.00&6697.44\\	
				
			(MP) & $|Vars|$&10278&10170&12096&14971&13627 &13725  \\
			
			& $|Nodes|$&1&1&2&55&1 &1  \\
			
			& $|Cuts|$&5436&5073&4734&7807&7907 &9183  \\

                   & $\# unsolved$&10&9&9&10&10 &9  \\

                   & $Gap(\%)$&5.54&4.36&3.72&3.96&5.06 &3.87  \\\hline
	&	& \multicolumn{3}{c|}{$n=60$} & \multicolumn{3}{c}{$n=70$}  \\
			&	& $p=15$ & $p=20$ & $p=30$ & $p=17$ & $p=23$ & $p=35$ \\ \hline	
			\textbf{B\&C} &   $Time(s)$&2768.88&3306.54&6707.38&1842.00&2119.13&2474.98\\
						
			(WOC) & $|Vars|$&161807&144983&109804&259406&231680 &173955  \\
			
			& $|Nodes|$&1&20330&85723&1&1 &835  \\
			
			& $|Cuts|$&18081&19887&15676&16115&23603 &19238  \\

                   & $\# unsolved (T/M)$&0/8&2/8&8/2&0/10&0/10 &2/8  \\

                   & $Gap(\%)$&2.74&2.86&1.78&5.67&5.77 &7.12  \\\hline

			\textbf{\BPC} & $Time(s)$&7200.00&6864.94&7200.00&7200.00&7200.00&7200.00\\	
				
			(MP) & $|Vars|$&21094&16077&17634&31949&32345&22175  \\
			
			& $|Cuts|$&12999&16061&13342&14722&20532 &19240  \\
			
			& $|Cuts|$&8917&13099&12406&5252&2058 &17238  \\

                   & $\# unsolved$&10&9&10&10&10 &10  \\

                   & $Gap(\%)$&7.04&6.83&7.48&6.95&8.14 &8.35  \\\hline
				&	& \multicolumn{3}{c|}{$n=80$} & \multicolumn{3}{c}{$n=90$}  \\
			&	& $p=20$ & $p=26$ & $p=40$ & $p=22$ & $p=30$ & $p=45$ \\ \hline	
			\textbf{B\&C} &   $Time(s)$&2902.00&2886.25&3428.13&5999.16&5214.89 &6243.49  \\
						
			(WOC) & $|Vars|$&383199&346926&259186&549561&488316 &368560  \\
			
			& $|Nodes|$&1&1&1&1&1 &1  \\
			
			& $|Cuts|$&27129&25187&12406&46216&32406 &12157  \\

                   & $\# unsolved (T/M)$&0/10&0/10&0/10&0/10&0/10 &7/3  \\

                   & $Gap(\%)$&6.50&5.28&3.26&6.37&4.42 &4.06  \\\hline

			\textbf{\BPC} & $Time(s)$&7200.00&7200.00&7200.00&7200.00&7200.00&7200.00\\	
				
			(MP) & $|Vars|$&41971&34634&17640&41239&36230 &23826  \\
			
			& $|Nodes|$&384&1196&1&294&625 &314  \\
			
			& $|Cuts|$ &27360&24059&13884&43721&31810&11061 \\

                   & $\# unsolved$&10&10&10&10&10 &10  \\

                   & $Gap(\%)$&8.33&7.09&3.16&8.37&6.14 &4.56  \\\hline
				&	& \multicolumn{3}{c|}{$n=100$}  \\
			&	& $p=25$ & $p=33$ & $p=50$  \\ \hline	
			\textbf{B\&C} &   $Time(s)$&--&--&--&  \\
						
			(WOC) & $|Vars|$&&&&& &  \\
			
			& $|Nodes|$&--&--&--&  \\
			
			& $|Cuts|$&--&--&--&  \\

                   & $\# unsolved (T/M)$&--&--&--&  \\

                   & $Gap(\%)$&--&--&--&  \\\hline

			\textbf{\BPC} & $Time(s)$&7200.00&7200.00&7200.00&&&\\	
				
			(MP) & $|Vars|$&40905&40552&31199&& &  \\
			
			& $|Nodes|$&319&389&68&& &  \\
			
			& $|Cuts|$&77889&54296&15408&& &  \\

                   & $\# unsolved$&10&10&10&& &  \\

                   & $Gap(\%)$&7.77&7.12&5.49&& &  \\\hline

		\end{tabular}
		\caption{Numerical results for \textbf{B\&C}(WOC) and \BPC(MP) } \label{Results20to100}
	\end{center}
\end{table}

The results are reported in Table \ref{Results20to100}. In that table, we denote by $Time(s)$ the average computational time (in seconds) required by each method to obtain an optimal solution for a given set of 10 instances defined by number of clients ($n$) and number of open facilities ($p$). We report 7200 s. in those cases where the optimal solution is not obtained in 2 hours.

With  $|Vars|$ we refer to  the average of the numbers of variables used by $MP$ or $WOC$. We also denote by $|Nodes|$ and $|Cuts|$ the average of the number of nodes explored and the average of the number of cuts used, respectively, in the corresponding methodology. The row $\#unsolved(T/M)$ in the case of \textbf{B\&C}(WOC) reports the number of unsolved instances out of the 10 in each group. It distinguishes between those instances not solved by exceeding the maximum running time ($T$) or the memory limits ($M$). Observe that in the similar row within the blocks \textbf{B\&P\&C}(MP) no distinction is shown since the memory limit is never reached and instances not solved are only due to time limitations. Finally, we also include in our report the gap at termination ($GAP(\%)$).

Analyzing further the results in Table \ref{Results20to100} we conclude that on average \textbf{B\&C}(WOC) is faster than \textbf{B\&P\&C}(MP). We could explain this behavior because of the professional implementation of Cplex to handle the branching tree and its sophisticated branching strategies that we cannot reproduce in our implementation.   On the other hand, remark the much smaller number of variables and thus, memory requirements, used by  \BPC (MP) as compared with \textbf{B\&C}(WOC). Actually, one of the most important features of our $MP$ formulation is that it needs much less number of variables than $WOC$, allowing us solving larger size instances with $MP$ that were not affordable for the original $WOC$.

We also observe that the number of required cuts for \BPC (MP) is smaller than for \textbf{B\&C}(WOC). This could be explained by the tightness of \BPC (MP) with respect to \textbf{B\&C}(WOC). After adding cuts \BPC (MP) is able to solve the problem in many of  the cases at the root node. This behavior does not occur for \textbf{B\&C}(WOC).
The number of instances solved to optimality, for small size instances up to $n=40$, is slightly better for \textbf{B\&C}(WOC). As the size increases this number is similar in both cases. Gaps at termination, after 7200 seconds, are always smaller than $8\%$ for \BPC (MP) and smaller than $7.15\%$ for \textbf{B\&C}(WOC), being the later slightly better. For the larger instances of $n=80,90$ gaps are  similar. Finally, \textbf{B\&C}(WOC) was not able to handle any instance with $n=100$ (reporting out of memory flags) whereas \BPC (MP) reports the same performance than for the  previous sizes.

To conclude, despite the promising better root node gap, and the features developed for \BPC(MP), such as the stabilization, hurry pricer, cuts, etc., the overall performance of this framework in solving DOMP is not systematically better than the branch-and-cut formulation \textbf{B\&C}(WOC). In small to medium size instances \textbf{B\&C}(WOC) is faster and achieves slightly smaller gaps. Nevertheless, in larger size instances performance is similar.  Moreover, as expected, we were able to handle the largest considered sizes $(n=100)$ only with \BPC (MP) and not with \textbf{B\&C}(WOC).

\clearpage
\section{Conclusions\label{section:5}}
This paper presents a first branch-price-and-cut, \BPC(MP), algorithm for solving DOMP. This approach is based on an extended formulation using an exponential number of variables coming from a set partitioning model. Elements in the partitions are couples containing information about a client and its sorted position in the sorted sequence of allocation costs. To address the solution of this formulation we develop a column generation algorithm and we prove that the pricing routine is polynomially solvable by a dynamic programming algorithm. We embed the column generation algorithm within a brand-and-price framework. Furthermore, we adapt preprocessing and incorporate families of valid inequalities that improve its performance. Extensive computational results compare the performance of our \BPC(MP) against the most recent algorithm in the literature for DOMP, \textbf{B\&C}(WOC), showing that for the largest considered instances \BPC(MP) performs better and it requires less memory to upload and run the models.

\section*{Acknowledgements}

This research has been partially supported by Spanish Ministry of Econom{\'\i}a and  Competitividad/FEDER grants number MTM2016-74983-C02-01. The research of the second and third authors was partially  supported by the Interuniversity Attraction Poles Programme initiated by the Belgian Science Policy Office. We thank the SCIP team (\cite{GamrathFischerGallyetal}) for the helpful advices.

\bibliographystyle{plainnat}
\bibliography{mybibDP}

\newpage

\appendix
\section{Appendix}

\subsection{GRASP\label{c3:ss35}}
In the following we report the detailed implementation of the functions \textit{ConstructGreedySolution} and \textit{LocalSearch} in the GRASP algorithm \ref{c3:al GRASP}.

\begin{algorithm}[H]
\begin{algorithmic}[1]
\small
\STATE Input($|J|=q\le p$);
\WHILE {$|J|< p$}
\STATE $j^*=\emptyset$;
\STATE $value=M$;
\FOR {$j\in \bar J$}
\IF {$z(J\cup \{j\})<value$}
\STATE $value=z(J\cup \{j\})$;
\STATE$j^*=\{j\}$;
\ENDIF
\ENDFOR
\STATE$J=J\cup\{j^*\}$;
\ENDWHILE
\end{algorithmic}
\caption{{ConstructGreedySolution}.\label{c3:al GRASP}}
\end{algorithm}

\begin{algorithm}[H]
\begin{algorithmic}[1]
\small
\STATE Input($|J|=p$);
\STATE $\bar z= z(J)$;
\FOR {$n_2$ iterations}
\FOR {$j_1\in J$}
\FOR {$j_2\in \bar J$}
\IF {$z((J\setminus\{j_1\})\cup \{j_2\})<\bar z$}
\STATE $\bar z = z((J\setminus\{j_1\})\cup \{j_2\})$
\STATE$J=(J\setminus\{j_1\})\cup \{j_2\}$
\ENDIF
\ENDFOR
\ENDFOR
\ENDFOR
\end{algorithmic}
\caption{{LocalSearch(Solution)}.\label{c3:al GRASP}}
\end{algorithm}

\subsection{Handling cuts within the Hurry pricer\label{ap:hp}}
The following algorithms try to avoid useless calculations in Algorithm \ref{hurryPricer} while we handle the $\zeta$ values (dual multipliers of the cuts). The idea is that, because the cuts are relatively rare, the $\zeta$ are often equal to zero. For example, in one of our experiments, we activated only 58 cuts among a maximum of 64 000, solving a $n=40$ instance.

We need to save the index for each new cut added. We note $ListOfBiIndex$ the sorted 3-tuple list of index ($c_i$,$c_j$,$c_k$) for each cut $c$. It is sorted by $k$ and then according to the costs. This list is updated after each separator has been called. Then, we can have several pricings using the same $ListOfBiIndex$, while the duals $\zeta$ are changing at each iteration.

We note $VVP$ the vector of vectors of pairs such that it saving the increasing and decreasing sums of $\zeta$. The increasing sums are accessible by $first$ and the decreasing sums by $second$. First, we fill out a data structure $VVP$ with the right sum for each individual tuple of index from $ListOfBiIndex$ and for all $k=1..n$ (cf. Algorithm \ref{FastSumsDualCutsValues}). Second, we finish to fill out $VVP$ for the other index with the existing source.

\begin{algorithm}
\caption{FastSumsDualCutsValues}
\begin{algorithmic}[1]
 \STATE Take the list of tuples $ListOfBiIndex$ from the last call of the Separator ;
 \STATE Take the duals $\zeta$ from the last restricted MP resolution ;
 \STATE $kPrevious = 0$ ; Initialize all $VVP$ with 0 ;
 \FOR{the 3-tuple ($(index = (i,j))$,$k$) in the normal order of $ListOfBiIndex$}
 	\IF{$kPrevious \neq k$}
 		\STATE $Previous = 0$ ;
 		\STATE $kPrevious = k$ ;		
 	\ENDIF
	\STATE $VVP[index][k].first = Previous + \zeta_{ij}^k$ ;
	\STATE $Previous = VVP[index][k].first$ ;
 \ENDFOR
 \STATE $kPrevious = 0$ ;
  \FOR{the 3-tuple ($(index_r = (i_r,j_r))$,$k_r$) in the reverse order of $ListOfBiIndex$}
  	\IF{$kPrevious \neq k_r$}
  		\STATE $Previous = 0$ ;
  		\STATE $kPrevious = k_r$ ;		
  	\ENDIF
 	\STATE $VVP[index_r][k_r].second = Previous + \zeta_{i_rj_r}^k$ ;
 	\STATE $Previous = VVP[index_r][k_r].second$ ;
  \ENDFOR
\STATE \RETURN $VVP$;
\end{algorithmic}\label{FastSumsDualCutsValues}
\end{algorithm}

This first algorithm will fill out the structure $VVP$ with the sums of the dual $\zeta$. $first$ gives the dimension saving the sums in the increasing order, in order to have directly the value $ \sum^n_{ i'=1} \sum^n_{\substack{ j'=1: \\ C_{ i' j'} \leq C_{i_lj}}} \zeta_{ i' j'}^k $  and $second$ determines the reverse order to obtain $ \sum^n_{ i'=1} \sum^n_{\substack{ j'=1: \\ C_{ i' j'} \geq C_{i_lj}}}  \zeta_{ i' j'}^{k-1}$ faster. The Algorithm \ref{SpreadSums} takes for input the $VPP$ updated from the last call of Algorithm \ref{FastSumsDualCutsValues}. It will copy the non-zero sums (so from the index for those we added a cut) to the other cells such that the value of the current cell (so with an "non-cut index") is equal, for the same $k$, to the last sum in the increasing or decreasing order (resp. for $first$ and $second$ dimensions).

\begin{algorithm}
\caption{SpreadSums}
\begin{algorithmic}[1]
 \FOR{$k=1..n$}
 	\STATE $Current = 0$ ;  $Current_r = 0$ ;
 	 \FOR{$index=1..n^2$}
 	 	\IF{$VVL[index][k].first \neq 0$}
 	 		\STATE $Current = VVL[index][k].first$ ;
 	 	\ELSE
 	 		\STATE $VVL[index][k].first = Current$ ;
 	 	\ENDIF
 	 	\STATE $index_r = 1 + n^2 - index$ ;
  	 	\IF{$VVL[index_r][k].second \neq 0$}
  	 		\STATE $Current_r = VVL[index_r][k].second$ ;
  	 	\ELSE
  	 		\STATE $VVL[index_r][k].second = Current_r$ ;
  	 	\ENDIF	 	
 	 \ENDFOR
 \ENDFOR
\STATE \RETURN $VVP$;
\end{algorithmic}\label{SpreadSums}
\end{algorithm}

We can now replace the time consuming instruction of the Algorithm \ref{hurryPricer}:

 ``$d_{i_lj}^k = d_{i_lj}^k + \sum^n_{ i'=1} \sum^n_{\substack{ j'=1: \\ C_{ i' j'} \leq C_{i_lj}}} \zeta_{ i' j'}^k + \sum^n_{ i'=1} \sum^n_{\substack{ j'=1: \\ C_{ i' j'} \geq C_{i_lj}}}  \zeta_{ i' j'}^{k-1}$ ; "

with the following instruction :

`` $d_{i_lj}^k = d_{i_lj}^k + VVP[index=(i,j)][k].first + VVP[index=(i,j)][k-1]].second$ ; "

\end{document}